\documentclass[11pt,a4paper]{article}
\usepackage{a4wide,amsfonts,amsmath,latexsym,amssymb,euscript,graphicx,units,mathrsfs}

\usepackage{amsmath, amsthm, amsfonts, amssymb,eufrak,color}

\usepackage{hyperref}

\usepackage{float}
\newfloat{figure}{H}{lof}
\floatname{figure}{\figurename}

\DeclareMathAlphabet{\eufrak}{U}{}{}{}  
\SetMathAlphabet\eufrak{normal}{U}{euf}{m}{n}
\SetMathAlphabet\eufrak{bold}{U}{euf}{b}{n}

\numberwithin{equation}{section}

\def\real{{\mathord{{\rm I\kern-2.8pt R}}}}        
\def\inte{{\mathord{{\rm I\kern-2.8pt N}}}}
\def\PP{{\mathord{{\rm I\kern-2.8pt P}}}}

\def\real{{\mathord{\mathbb R}}}

\def\inte{{\mathord{\mathbb N}}}

\def\H{{\mathord{\mathbb H}}}

\def\R{\right}
\def\L{\left}

\def\P{\mathbb{P}}
\def\E{\mathbb{E}}
\def\Q{\mathbb{Q}}

\def\L{\left}
\def\R{\right}

\def\H{\mathcal{H}}

\newtheorem{prop}{Proposition}[section]

\newtheorem{lemma}[prop]{Lemma}

\newtheorem{corollary}[prop]{Corollary}
\newtheorem{theorem}[prop]{Theorem}
\newtheorem{remark}[prop]{Remark}

\def\argmax{{\mathrm{{\rm argmax}}}}

\def\cY{\mathcal{Y}}
\def\cF{\mathcal{F}}
\def\IP{\mathbb{P}}

\def\cH{\mathcal{H}}
\def\cO{\mathcal{O}}
\def\cE{\mathcal{E}}

\def\cS{\mathbb{S}}

\def\N{\mathbb{N}}
\def\IR{\mathbb{R}}

\title{Forward-backward systems for expected utility maximization}

\author{
Ulrich Horst\footnote{Institut f\"ur Mathematik, Humboldt-Universit\"at zu Berlin, Unter den Linden 6, 10099 Berlin, Germany, \texttt{horst@mathematik.hu-berlin.de}}, Ying Hu\footnote{Universit\'e de Rennes 1, campus Beaulieu, 35042 Rennes cedex, France, \texttt{ying.hu@univ-rennes1.fr}}, Peter Imkeller\footnote{Institut f\"ur Mathematik, Humboldt-Universit\"at zu Berlin, Unter den Linden 6, 10099 Berlin, Germany, \texttt{imkeller@mathematik.hu-berlin.de}}, Anthony R\'eveillac\footnote{CEREMADE UMR CNRS 7534, Universit\'e Paris Dauphine, Place du Mar\'echal De Lattre De Tassigny, 75775 PARIS CEDEX 16, France, \texttt{anthony.reveillac@cremade.dauphine.fr}} and Jianing Zhang\footnote{Weierstrass Institute for Applied Analysis and Stochastics, Mohrenstr. 39, 10117 Berlin, Germany, \texttt{jianing.zhang@wias-berlin.de}}  \vspace{0.3cm}
}

\allowdisplaybreaks

\begin{document}
\maketitle

\begin{abstract}
In this paper we deal with the utility maximization problem with a general utility function. We derive a new approach in which we reduce the utility maximization problem with general utility to the study of a fully-coupled Forward-Backward Stochastic Differential Equation (FBSDE).
\end{abstract}

\mbox{ }

\noindent \textbf{AMS Subject Classification}: Primary 60H10, 93E20

\noindent \textbf{JEL Classification}: C61, D52, D53

\section{Introduction}

One of the most commonly studied topic in mathematical finance (and applied probably) is the problem of maximizing expected terminal utility from trading in a financial market. In such a situation, the \textit{stochastic control problem} is of the form
\begin{equation}
\label{eq:introopti}
V(0,x):=\sup_{\pi \in \mathcal{A}} \E[U(X_T^\pi+H)]
\end{equation}
for a real-valued function $U$, where $\mathcal{A}$ denotes the set of {\sl admissible trading strategies},  $T < \infty$ is the terminal time, $X_T^{\pi}$ is the wealth of the agent when he follows the strategy $\pi \in \mathcal{A} $ and his initial capital at the initial time zero is $x>0$, and $H$ is a liability that the agent must deliver at the terminal time. One is typically interested in establishing existence and uniqueness of optimal solutions and in characterizing optimal strategies and the \textit{value function} $V(t,x)$ which is defined as
$$ V(t,x):=\sup_{\pi \in \mathcal{A}} \E[U(X_{t,T}^\pi+H)\vert \mathcal{F}_t].$$
Here $X_{t,T}$ denotes the wealth of the agent when the investment period is $[t,T]$ and where the filtration $(\mathcal{F}_t)_{t\in [0,T]}$ defines the flow of information.
\\[5pt]
The question of existence of an optimal strategy $\pi^\ast$ can essentially be addressed using \textit{convex duality}. The convex duality approach is originally due to Bismut \cite{Bismut} with its  modern form dating back to Kramkov and Schachermayer \cite{KramkovSchachermayer}. For instance, given some growth condition on $U$ or related quantities (such as the asymptotic elasticity condition for utilities defined on the half line) existence of an optimal strategy is guaranteed under mild regularity conditions on the liability and convexity assumptions on the set of admissible trading strategies (see e.g. \cite{Biagini} for details). However, the duality method is not constructive and does not allow for a characterization of optimal strategies and value functions.
\\[5pt]
One approach to simultaneously characterize optimal trading strategies and utilities uses the theory of forward-backward stochastic differential equations (FBSDE). When the filtration is generated by a standard Wiener process $W$ and if either $U(x):=-\exp(-\alpha x)$ for some $\alpha>0$ and $H \in L^2$, or $U(x):=\frac{x^\gamma}{\gamma}$ for $\gamma \in (0,1)$ or $U(x) = \ln x$ and $H = 0$, it has been shown by Hu, Imkeller and M\"uller \cite{HuImkellerMueller} that the control problem \eqref{eq:introopti} can essentially be reduced to solving a BSDE of the form
\begin{equation}
\label{eq:introBSDE}
Y_t =H-\int_t^T Z_s dW_s-\int_t^T f(s,Z_s) ds, \quad t\in [0,T],
\end{equation}
where the {\sl driver} $f(t,z)$ is a predictable process of quadratic growth in the $z$-variable. Their results have since been extended beyond the Brownian framework and to more general utility optimization problems with complete and incomplete information in, e.g., \cite{HorstPirvuReis}, \cite{MochaWestray}, \cite{Morlais}, \cite{Nutz} and \cite{ManiaSantacroce}. The method used in \cite{HuImkellerMueller} and essentially all other papers relies on the martingale optimality principle and can essentially only be applied to the standard cases mentioned above (exponential with general endowment and power, respectively logarithmic, with zero endowment). This is due to a particular ``separation of variables'' property enjoyed by the classical utility functions: their value function can be decomposed as $V(t,x)=g(x) V_t$ where $g$ is a deterministic function and $V$ is an adapted process. As a result, optimal future trading strategies are independent of  current wealth levels.
\\[5pt]
More generally, there has recently been an increasing interest in dynamic {\sl translation invariant utility functions}. A utility function is called translation invariant if a cash amount added to a financial position increases the utility by that amount and hence optimal trading strategies are wealth-independent\footnote{It has been shown by \cite{DelbaenPengRosazzaGianin} that essentially all such utility functions can be represented in terms of a BSDE of the form \ref{eq:introBSDE}.}. Although the property of translation invariance renders the utility optimization problem mathematically tractable, independence of the trading strategies on wealth is rather unsatisfactory from an economic point of view. In \cite{ManiaTevzadze} the authors derive a verification theorem for optimal trading strategies for more general utility functions when $H=0$. More precisely, given a general utility function $U$ and assuming that there exists an optimal strategy regular enough such that the value function enjoys some regularity properties in $(t,x)$, it is shown that there exists a predictable random field $(\varphi(t,x))_{(t,x)\in [0,T]\times (0,\infty)}$ such that the pair $(V,\varphi)$ is solution to the following backward stochastic partial differential equation (BSPDE) of the form:
\begin{equation}
\label{eq:introBSPDE}
V(t,x) =U(x)-\int_t^T \varphi(s,x) dW_s-\int_t^T \frac{|\varphi_x(s,x)|^2}{V_{xx}(s,x)} ds, \quad t\in [0,T]
\end{equation}
where $\varphi_x$ denotes the partial derivative of $\varphi$ with respect to $x$ and $V_{xx}$ the second partial derivative of $V$ with respect to the same variable. The optimal strategy $\pi^\ast$ can then be obtained from $(V,\varphi)$. Unfortunately, the BSPDE-theory is still in its infancy and to the best of our knowledge the non-linearities arising in \eqref{eq:introBSPDE} cannot be handled except in the classical cases mentioned above where once again one benefits of the ``separation of variables'' (see \cite{ImkellerReveillacZhang}). Moreover, the utility function $U$ only appears in the terminal condition which is not very handy. In that sense this is exactly the same situation as the Hamilton-Jacobi-Bellman equation where $U$ only appears as a terminal condition but not in the equation itself.
\\[5pt]
In this paper we propose a new approach to solving the optimization problem \eqref{eq:introopti} for a larger class of utility function and characterize the optimal strategy $\pi^\ast$ in terms of a fully-coupled FBSDE-system. The optimal strategy is then a function of the current wealth and of the solution to the backward component of the system. In addition, the driver of the backward part is given in terms of the utility function and its derivatives. This adds enough structure to the optimization problem to deal with fairly general utilities functions, at least when the market is complete. We also derive the FBSDE system for the power case with general (non-hedgeable) liabilities; to the best of our knowledge we are the first to characterize optimal strategies for power utilities with general liabilities. Finally, we link our approach to the well established approaches using convex dual theory and stochastic maximum principles.
\\[5pt]
The remainder of this paper is organized as follows. In Section \ref{section:prelim} we introduce our financial market model. In Section \ref{section:realline} we first derive a verification theorem in terms of a FBSDE for utilities defined on the real line along with a converse result, that is, we show that a solution to the FBSDE allows to construct the optimal strategy. Section \ref{section:halfline} is devoted to the same question but for utilities defined on the positive half line. In Section \ref{section:link} we relate our approach to the stochastic maximum principle obtained by Peng \cite{Peng} and the standard duality approach. We use the duality-BSDE link to show that the FBSDE associated with the problem of maximizing power utility with general positive endowment has a solution.

\section{Preliminaries}
\label{section:prelim}

We consider a financial market which consists of one bond $S^0$ with interest rate zero and of $d\geq 1$ stocks given by
$$ d\tilde{S}_t^i:=\tilde{S}_t^i dW^i_t + \tilde{S}_t^i \theta_t^i dt, \quad i \in \{1,\ldots,d\} $$
where $W$ is a standard Brownian motion on $\real^d$ defined on a filtered probability space $(\Omega,\mathcal{F},(\mathcal{F}_t)_{t\in [0,T]},\P)$, $(\mathcal{F}_t)_{t\in [0,T]}$ is the filtration generated by $W$, and $\theta:=(\theta^1,\ldots,\theta^d)$ is a predictable bounded process with values in $\real^d$. Since we assume the process $\theta$ to be bounded, Girsanov's theorem implies that the set of equivalent local martingale measures (i.e. probability measures under which $\tilde{S}$ is a local martingale) is not empty, and thus according to the classical literature (see e.g. \cite{DelbaenSchachermayer}), arbitrage opportunities are excluded in our model. 
For simplicity throughout we write
$$ dS_t^i:=\frac{d\tilde{S}_t^i}{\tilde{S}_t^i}. $$
We denote by $\alpha \cdot \beta$ the inner product in $\real^d$ of vectors $\alpha$ and $\beta$ and by $|\cdot|$ the usual associated $L^2$-norm on $\real^d$. In all the paper $C$ will denote a generic constant which can differ from line to line. We also define the following spaces:
$$ \mathbb{S}^2(\real^d):=\left\{\beta:\Omega \times [0,T] \to \real^d, \; \textrm{ predictable}, \; \E[\sup_{t\in [0,T]} |\beta_t|^2]<\infty \right\}, $$
$$ \mathbb{H}^2(\real^d):=\left\{\beta:\Omega \times [0,T] \to \real^d, \; \textrm{ predictable}, \; \E\left[\int_0^T |\beta_t|^2 dt\right]<\infty \right\}.$$
\\\\
\noindent
Since the market price of risk $\theta$ is assumed to be bounded, the stochastic process
$$ \mathcal{E}(-\theta \cdot W)_t:=\exp\left(-\int_0^t \theta_s dW_s - \frac12 \int_0^t |\theta_s|^2 ds \right) $$
has finite moments of order $p$ for any $p>0$.
We assume $d_1+d_2=d$ and that the agent can invest in the assets $\tilde{S}^1,\ldots,\tilde{S}^{d_1}$ while the stocks $\tilde{S}^{d_1+1},\ldots,\tilde{S}^{d_2}$ cannot be invested into. Denote $S^{\mathcal{H}}:=(S^1,\ldots,S^{d_1},0\ldots,0)$, $W^{\mathcal{H}}:=(W^1,\ldots,W^{d_1},0\ldots,0)$, $W^{\mathcal{O}}:=(0,\ldots,0,W^{d_1+1},\ldots,W^{d_2}),$ and $\theta^{\mathcal{H}}:=(\theta^1,\ldots,\theta^{d_1},0\ldots,0)$ (the notation $\cH$ refers to ``hedgeable'' and $\cO$ to ``orthogonal''). We define the set $\Pi^x$ of admissible strategies with initial capital $x>0$ as\begin{equation}
\label{eq:admissible}
\Pi^x:=\L\{ \pi:\Omega \times [0,T] \to \real^{d_1}, \; \E\left[ \int_0^T |\pi_t|^2 dt \right] <\infty, \pi \textrm{ is self-financing }\R\}
\end{equation}
where for $\pi$ in $\Pi^x$ the associated wealth process $X^\pi$ is defined as
$$ X_t^\pi:=x+\int_0^t \pi_r dS_r^{\mathcal{H}} = x+ \sum_{i=1}^{d_1} \int_0^t \pi_r^i dS_r^i, \quad t \in [0,T].$$
Every $\pi$ in $\Pi^x$ is extended to an $\IR^d$-valued process by
$$\tilde{\pi}:=(\pi^1,\ldots,\pi^{d_1},0,\ldots,0).$$ In the following, we will always write $\pi$ in place of $\tilde{\pi}$, i.e. $\pi$ is an $\IR^d$-valued process where the last $d_2$ components are zero. Moreover, we consider a utility function $U:I \to \real$ where $I$ is an interval of $\real$ such that $U$ is strictly increasing and strictly concave.
We seek for a strategy $\pi^\ast$ in $\Pi^x$ satisfying $\E[U(X_T^{\pi^\ast}+H)]<\infty$ such that
\begin{equation}
\label{eq:opti}
\pi^\ast=\argmax_{\pi \in \Pi^x,\; \E[\vert U(X_T^{\pi}+H) \vert]<\infty} \L\{ \E[U(X_T^{\pi}+H)]\R\}
\end{equation}
where $H$ is a random variable in $L^2(\Omega,\mathcal{F}_T,\P)$ such that the expression above makes sense. We concretize on sufficient conditions in the subsequent sections.

\section{Utilities defined on the real line}
\label{section:realline}

In this section we consider a utility function $U:\real \to \real$ defined on the whole real line. We assume that $U$ is strictly increasing and strictly concave and that the agent is endowed with a claim $H\in L^2(\Omega,\mathcal{F}_T,\P)$. We introduce the following conditions.
\\\\\textbf{(H1)} $U:\real \to \real$ is three times differentiable\\\\
\textbf{(H2)} We say that condition (H2) holds for an element $\pi^\ast$ in $\Pi^x$, if $\E[|U'(X_T^{\pi^\ast}+H)|^2]<\infty$ and if for every bounded predictable process $h:[0,T] \to \real$, the family of random variables
$$ \left(\int_0^T h_r dS_r^{\mathcal{H}} \int_0^1 U'\left( X_T^{\pi^\ast}+H+ \varepsilon r \int_0^T h_r dS_r^{\mathcal{H}} \right) dr\right)_{\varepsilon \in (0,1)}  $$
is uniformly integrable.\\

\noindent
Before presenting the first main result of this section, we prove that condition $(H2)$ is satisfied for every strategy $\pi^\ast$ such that $\E[|U'(X_T^{\pi^\ast}+H)|]<\infty$ when one has an exponential growth condition on the marginal utility of the form:
\[
    U'(x + y) \leq C \left( 1 + U'(x) \right) \left( 1 + \exp(\alpha y) \right) \quad \mbox{for some } \alpha \in \mathbb{R}.
\]
Indeed, let
$G:=\int_0^T h_r dS_r^{\cH}$ and $d>0$. We will show that the quantity
\begin{align*}
q(d):=\sup_{\varepsilon\in (0,1)} \E\left[ \left|G \int_0^1 U'(X_T^{\pi^\ast}+H+\varepsilon r G) dr \right| \textbf{1}_{\left|G \int_0^1 U'(X_T^{\pi^\ast}+H+\varepsilon r G) dr \right|>d}\right]
\end{align*}
vanishes when $d$ goes to infinity. For simplicity we write $\delta_{\varepsilon,d}:=\textbf{1}_{\left|G \int_0^1 U'(X_T^{\pi^\ast}+H+\varepsilon r G) dr \right|>d}$. By the Cauchy-Schwarz inequality
\begin{align*}
q(d)&\leq \sup_{\varepsilon\in (0,1)} \E\left[ (1+U'(X_T^{\pi^\ast}+H)) \left|G (1 + \int_0^1  \exp(\alpha \varepsilon r G) ) dr \right| \delta_{\varepsilon,d}\right]\\
& \leq C \E\left[ |U'(X_T^{\pi^\ast}+H)|^2\right]^{1/2} \sup_{\varepsilon\in (0,1)} \E\left[ \left|G \int_0^1 \exp(\alpha \varepsilon r G) dr \right|^2 \delta_{\varepsilon,d} \right]^{1/2}.
\end{align*}
Since $\E\left[ |U'(X_T^{\pi^\ast}+H)|^2\right]$ is assumed to be finite we deduce from the inequality
\[
    \exp(\alpha \zeta x) \leq 1 + \exp(\alpha x) \quad \mbox{for all } x\in \real, \; 0<\zeta<1
\]
that
\begin{align*}
q(d)&  \leq C \sup_{\varepsilon\in (0,1)} \E\left[ \left| G (2+\exp(\alpha G) ) \right|^2 \delta_{\varepsilon,d} \right]^{1/2}.
\end{align*}
Applying successively the Cauchy-Schwarz inequality and the Markov inequality, it holds that
\begin{align*}
q(d)&  \leq C \E\left[ \left|G  (2+ \exp(\alpha G)) \right|^4\right]^{1/4} \sup_{\varepsilon\in (0,1)} \E[\delta_{\varepsilon,d}]^{1/4}\\
&  \leq C \E\left[ \left|G  (2 + \exp(\alpha G)) \right|^4\right]^{1/4} d^{-1/4} \sup_{\varepsilon\in (0,1)} \E\left[|G| \int_0^1 U'(X_T^{\pi^\ast}+H+\varepsilon r G) dr \right]^{1/4}\\
&  \leq C \E\left[ \left|G  (2+ \exp(\alpha G)) \right|^4\right]^{1/4} d^{-1/4} \;\E\left[|G (2+\exp(\alpha G))|^2 \right]^{1/8}.
\end{align*}
Let $p\geq 2$. Since $h$ and $\theta$ are bounded it is clear that $\E\left[ |G|^{2p}\right]<\infty$ and
\begin{align*}
& \E\left[ \left|G  (2+ \exp(\alpha G)) \right|^p\right]\\
&\leq \E\left[ |G|^{2p}\right]^{1/2}  \E\left[\left|2+ \exp(\alpha G)\right|^{2p}\right]^{1/2}\\
& \leq C \left(2+\E\left[\left|\exp(\alpha G)\right|^{2p}\right]\right)^{1/2}\\
&=C \left(2+\E\left[\exp\left(\int_0^T 2p \alpha h_r dW_r^\cH - \frac12 \int_0^T |2p \alpha h_r|^2 dr\right)\right. \right. \\
 & ~~~ \left. \left. \exp\left(\frac12 \int_0^T |2p \alpha h_r|^2 +2p \alpha h_r \cdot \theta_r dr\right) \right]\right)^{1/2} \\
&\leq C.
\end{align*}
Hence $\lim_{d\to \infty} q(d)=0$ which proves the assertion.

\noindent

\subsection{Characterization and verification: incomplete markets}

We are now ready to state and prove the first main result of this paper: a verification theorem for optimal trading strategies.

\begin{theorem}
\label{th:verif1}
Assume that $(H1)$ holds. Let $\pi^\ast \in \Pi^x$ be an optimal solution to the problem \eqref{eq:opti} which satisfies assumption $(H2)$. Then there exists a predictable process $Y$ with $Y_T=H$ such that $U'(X^{\pi^\ast}+Y)$ is a martingale in $L^2(\Omega,\mathcal{F}_T,\P)$ and $$ \pi_t^{\ast^i}= -\theta_t^i \frac{U'(X_t^{\pi^\ast}+Y_t)}{U''(X_t^{\pi^\ast}+Y_t)}- Z_t^i, \quad t \in [0,T], \quad i=1,\ldots,d_1$$ where $Z_t:=\frac{d\langle Y, W \rangle_t}{dt}:=\left(\frac{d\langle Y, W^i \rangle_t}{dt},\ldots,\frac{d\langle Y, W^d \rangle_t}{dt}\right)$.
\end{theorem}

\begin{proof}
We first prove the existence of $Y$. Since $\E[|U'(X_T^{\pi^\ast}+H)|^2]<\infty$, the stochastic process $\alpha$ defined as $\alpha_t:=\E[U'(X_T^{\pi^\ast}+H)\vert \mathcal{F}_t]$, for $t$ in $[0,T]$ is a square integrable martingale. Define $Y_t:=(U')^{-1}(\alpha_t)-X^{\pi^\ast}_t$. Then $Y$ is $(\mathcal{F}_t)_{t\in [0,T]}$-predictable. Now It\^o's formula yields
\begin{align}
\label{eq:Y1}
Y_t+X_t^{\pi^\ast}&=Y_T+X_T^{\pi^\ast}-\int_t^T \frac{1}{U''(U'^{-1}(\alpha_s))} d\alpha_s + \frac12 \int_t^T \frac{U^{(3)}(U'^{-1}(\alpha_s))}{(U''(U'^{-1}(\alpha_s)))^3} d \langle \alpha, \alpha \rangle_s.
\end{align}
By definition, $\alpha$ is the unique solution of the zero driver BSDE
\begin{equation}
\label{eq:alphabsde}
\alpha_t=U'(X^{\pi^\ast}_T+Y_T) -\int_t^T \beta_s dW_s, \quad t \in [0,T],
\end{equation}
where $\beta$ is a square integrable predictable process with valued in $\real^d$.
Plugging \eqref{eq:alphabsde} into \eqref{eq:Y1} yields
\begin{align*}
Y_t+X_t^{\pi^\ast}=&X_T^{\pi^\ast}+ H - \int_t^T \frac{1}{U''(X_s^{\pi^\ast}+Y_s))} \beta_s dW_s + \frac12 \int_t^T \frac{U^{(3)}(X_s^{\pi^\ast}+Y_s)}{(U''(X_s^{\pi^\ast}+Y_s))^3} |\beta_s|^2 ds.
\end{align*}
Setting $\tilde{Z}:=\frac{1}{U''(X^{\pi^\ast}+Y))} \beta$, we have
\begin{align*}
Y_t+X_t^{\pi^\ast}=&X_T^{\pi^\ast}+H - \int_t^T \tilde{Z}_s dW_s + \frac12 \int_t^T \frac{U^{(3)}}{U''}(X_s^{\pi^\ast}+Y_s) |\tilde{Z}_s|^2 ds.
\end{align*}
Now by putting $Z^i:=\tilde{Z}^i-\pi{^\ast}^i, \; i=1,\ldots,d$, we have shown that $Y$ is a solution to the BSDE
\begin{equation}
\label{eq:BSDE1}
Y_t=H-\int_t^T Z_s d W_s -\int_t^T f(s,X_s^{\pi^\ast},Y_s,Z_s) ds, \quad t \in [0,T],
\end{equation}
where $f$ is given by
\begin{equation}
\label{eq:driver}
f(s,X_s^{\pi^\ast},Y_s,Z_s):= - \frac12 \frac{U^{(3)}}{U''}(X_s^{\pi^\ast}+Y_s) |\pi_s^\ast+Z_s|^2 - \pi_s^\ast \cdot \theta_s.
\end{equation}
Finally, by construction we have $U'(X_t^{\pi^\ast}+Y_t)=\alpha_t$, thus it is a martingale.
\\\\
\noindent
Now we deal with the characterization of the optimal strategy. To this end, let $h:[0,T] \to \real^{d_1}$ be a bounded predictable process. We extend $h$ into $\real^d$ by setting $\tilde{h}:=(h^1,\ldots,h^{d_1},0,\ldots,0)$ and use the convention that $\tilde{h}$ is again denoted by $h$. Thus for every $\varepsilon$ in $(0,1)$ the perturbed strategy $\pi^\ast + \varepsilon h$ belongs to $\Pi^x$. Since $\pi^\ast$ is optimal it is clear that for every such $h$
it holds that
\begin{equation}
\label{eq:direct1}
l(h):=\lim_{\varepsilon \to 0} \frac{1}{\varepsilon} \E\L[ U(x+\int_0^T (\pi_r^\ast+ \varepsilon h_r) dS_r^{\mathcal{H}}+Y_T)-U(x+\int_0^T \pi_r^\ast dS_r^{\mathcal{H}} +Y_T) \R]\leq 0.
\end{equation}
Moreover we have
\begin{align*}
&\frac{1}{\varepsilon} \left(U(x+\int_0^T (\pi_r^\ast+ \varepsilon h_r) dS_r^{\mathcal{H}}+Y_T)-U(x+\int_0^T \pi_r^\ast dS_r^{\mathcal{H}}+Y_T) \right)\\
&=\int_0^T h_r dS_r^{\mathcal{H}}  \int_0^1 U'\left(X_T^{\pi^\ast}+Y_T+\theta \varepsilon \int_0^T h_r dS_r^{\mathcal{H}} \right) d\theta.
\end{align*}
Now using (H2), Lebesgue's dominated convergence theorem implies that \eqref{eq:direct1} can be rewritten as
\begin{equation}
\label{eq:direct2}
\E\L[ U'(X_T^{\pi^\ast}+Y_T) \int_0^T h_r dS_r^{\mathcal{H}} \R]\leq 0
\end{equation}
for every bounded predictable process $h$.
Applying integration by parts to $U'(X_s^{\pi^\ast} + Y_s)_{s\in[0,T]}$ and $\L(\int_0^s h_r dS_r^\cH \R)_{s\in[0,T]}$, we get
\begin{align*}
&U'(X_T^{\pi^\ast}+Y_T) \int_0^T h_r dS_r^{\mathcal{H}} \nonumber\\
&= U'(x+Y_0) \times 0 + \int_0^T U'(X_s^{\pi^\ast}+Y_s) h_s dS_s^{\mathcal{H}} \nonumber\\
&\quad + \int_0^T \int_0^s h_r dS_r^{\mathcal{H}} \; U''(X_s^{\pi^\ast}+Y_s) \L[ (\pi_s^\ast + Z_s) dW_s^{\mathcal{H}} + (\pi_s^\ast \cdot \theta_s + f(s,X_s^{\pi^\ast},Y_s,Z_s))ds\R]  \nonumber\\
&\quad + \frac12 \int_0^T \int_0^s h_r dS_r^{\mathcal{H}} \; U^{(3)}(X_s^{\pi^\ast}+Y_s) |\pi_s^\ast + Z_s|^2 ds \nonumber\\
&\quad + \int_0^T U''(X_s^{\pi^\ast}+Y_s) h_s \cdot (\pi_s^\ast + Z_s) ds.
\end{align*}
By definition of the driver $f$, the previous expression reduces to
\begin{align}
\label{eq:coml(h)}
&U'(X_T^{\pi^\ast}+Y_T) \int_0^T h_r dS_r^\cH \nonumber\\
&=\int_0^T \left( U'(X_s^{\pi^\ast}+Y_s) \theta_s + U''(X_s^{\pi^\ast}+Y_s) (\pi_s^\ast+Z_s)\right) \cdot h_s ds\nonumber\\
&\quad + \int_0^T \int_0^s h_r dS_r^{\mathcal{H}} \; U''(X_s^{\pi^\ast}+Y_s) (\pi_s^\ast + Z_s) dW_s^{\mathcal{H}} + \int_0^T U'(X_s^{\pi^\ast}+Y_s) h_s dW_s^{\mathcal{H}}.
\end{align}
The next step would be to apply the conditional expectations in \eqref{eq:coml(h)}, however the two terms on the second line of the right hand side are a priori only local martingales.
We start by showing that the first one is a uniformly integrable martingale. Indeed, from the computations which have led to \eqref{eq:BSDE1} we have that
$$ U''(X^{\pi^\ast}+Y) (\pi^\ast+Z)=\beta, $$
where we recall that $\beta$ is the square integrable process appearing in \eqref{eq:alphabsde}.
Using the BDG inequality we get
\begin{align*}
&\E\left[ \sup_{s\in [0,T]} \left\vert \int_0^s \int_0^r h_u dS_u^{\mathcal{H}} \; U''(X_r^{\pi^\ast}+Y_r) (\pi_r^\ast + Z_r) dW_r^{\mathcal{H}} \right\vert\right]\\
&\leq C \E\left[ \left\vert \int_0^T \left\vert \int_0^s h_r dS_r^{\mathcal{H}}\right\vert^2 |\beta_s|^2 ds \right\vert^{1/2}\right]\\
&\leq C \E\left[ \left(\sup_{s\in [0,T]} \left\vert \int_0^s h_r dS_r^{\mathcal{H}}\right\vert^2\right)^{1/2} \left(\int_0^T |\beta_s|^2 ds\right)^{1/2} \right].
\end{align*}
Young's inequality furthermore yields
\begin{align*}
&\E\left[ \left(\sup_{s\in [0,T]} \left\vert \int_0^s h_r dS_r^{\mathcal{H}}\right\vert^2\right)^{1/2} \left(\int_0^T |\beta_s|^2 ds\right)^{1/2} \right]\\
&\leq C \E\left[ \sup_{s\in [0,T]} \left\vert \int_0^s h_r dS_r^{\mathcal{H}}\right\vert^2\right] + C \E\left[\int_0^T |\beta_s|^2 ds \right]\\
&\leq C \left(1+\E\left[ \sup_{s\in [0,T]} \left\vert \int_0^s h_r dW_r^{\mathcal{H}}\right\vert^2\right] \right)
\end{align*}
where we have used that $h$ and $\theta$ are bounded. Applying once again the BDG inequality, we obtain
$$ \E\left[ \sup_{s\in [0,T]} \left\vert \int_0^s h_r dW_r^{\mathcal{H}}\right\vert^2\right] \leq 4 \E\left[ \int_0^T |h_r|^2 dr\right] <\infty.$$Putting together the previous steps, we have that
$$ \E\left[\sup_{s\in [0,T]} \left\vert \int_0^s \int_0^r h_u dS_u^{\mathcal{H}} \; U''(X_r^{\pi^\ast}+Y_r) (\pi_r^\ast + Z_r) dW_r^{\mathcal{H}} \right\vert\right]<\infty, $$
thus
we get
$$\E\left[ \int_0^T \int_0^s h_r dS_r^{\mathcal{H}} \; U''(X_s^{\pi^\ast}+Y_s) (\pi_s^\ast + Z_s) dW_s^{\mathcal{H}} \right]=0.$$

Note that $\left(\int_0^t U'(X_s^{\pi^\ast}+Y_s) h_s dW_s^\cH\right)_{t\in [0,T]}$ is a square integrable martingale. Indeed $U'(X^{\pi^\ast}+Y)=\alpha$ is a square integrable martingale and thus
$$ \E\left[ \int_0^T \left\vert U'(X_s^{\pi^\ast}+Y_s) h_s \right\vert^2 ds \right]<\infty.$$
Similarly,
$$ \E\left[\left\vert U'(X_T^{\pi^\ast}+Y_T) \int_t^T h_r dS_r^\cH \right\vert\right]<\infty. $$
Taking expectation in \eqref{eq:coml(h)} we obtain for every $n \geq 1$ that
\begin{align}
\label{eq:eq:coml(h)bis}
&\E\left[U'(X_T^{\pi^\ast}+Y_T) \int_0^T h_r dS_r^{\mathcal{H}} \right]\nonumber\\
&= \E\left[ \int_0^T \left( U'(X_s^{\pi^\ast}+Y_s) \theta_s + U''(X_s^{\pi^\ast}+Y_s) (\pi_s^\ast+Z_s) \right) \cdot h_s ds \right],
\end{align}
which in conjunction with \eqref{eq:direct2} leads to
\begin{equation*}
\E\left[ \int_0^T \left( U'(X_s^{\pi^\ast}+Y_s) \theta_s + U''(X_s^{\pi^\ast}+Y_s) (\pi_s^\ast+Z_s)\right) \cdot h_s ds \right]\leq 0
\end{equation*}
for every bounded predictable process $h$. Replacing $h$ by $-h$, we get
\begin{equation}
\label{eq:direct3}
\E\left[ \int_0^T \left( U'(X_s^{\pi^\ast}+Y_s) \theta_s + U''(X_s^{\pi^\ast}+Y_s) (\pi_s^\ast+Z_s)\right) \cdot h_s ds \right] = 0.
\end{equation}
Now fix $i$ in $\{1,\ldots,d_1\}$. Let $A_s^i:=U'(X_s^{\pi^\ast}+Y_s) \theta_s + U''(X_s^{\pi^\ast}+Y_s) (\pi_s^{\ast^i}+Z_s^i)$ and $h_s:=(0,\ldots,0,\textbf{1}_{A_s^i>0},0,\ldots,0)$ where the non-vanishing component is the $i$-th component.
From \eqref{eq:direct3} we get that
\begin{align*}
\E\L[\int_0^T \textbf{1}_{A_s^i>0} [U'(X_s^{\pi^\ast}+Y_s) \theta_s^i + U''(X_s^{\pi^\ast}+Y_s) (\pi_s^{\ast^i} + Z_s^i)] ds \R] = 0.
\end{align*}
Hence, $A^i\leq 0$, $d\P \otimes dt-a.e.$. Similarly by choosing $h_s=(0,\ldots,0,\textbf{1}_{A_s^i<0},0,\ldots,0)$ we deduce that
$$U'(X^{\pi^\ast}+Y) \theta^i + U''(X^{\pi^\ast}+Y_t) (\pi_t^{\ast^i} + Z_t^i) = 0, \quad d\P\otimes dt-a.e.$$ This concludes the proof since $i \in \{1,\ldots,d_1\}$ is arbitrary.
\end{proof}

The verification theorem above can also be expressed in terms of a fully-coupled Forward-Backward system.

\begin{theorem}
\label{th:verif2}
Under the assumptions of Theorem \ref{th:verif1}, the optimal strategy $\pi^\ast$ for \eqref{eq:opti} is given by
$$ \pi_t^{\ast^i}= -\theta_t^i \frac{U'(X_t+Y_t)}{U''(X_t+Y_t)}- Z_t^i, \quad t \in [0,T], \quad i=1,\ldots,d_1, $$
where $(X,Y,Z) \in \real\times\real\times\real^d$ is a triple of adapted processes which solves the FBSDE
\begin{equation}
\label{eq:fbsde1}
\left\lbrace
\begin{array}{lll}
X_t & = & x-\int_0^t \left( \theta_s \frac{U'(X_s+Y_s)}{U''(X_s+Y_s)} + Z_s\right) dW_s^{\mathcal{H}} -\int_0^t \left( \theta_s \frac{U'(X_s+Y_s)}{U''(X_s+Y_s)} + Z_s\right) \cdot \theta_s^{\mathcal{H}} ds\\
\\
Y_t & = & H-\int_t^T Z_s dW_s -\int_t^T \left[ -\frac12 |\theta_s^{\mathcal{H}}|^2 \frac{U^{(3)}(X_s+Y_s) |U^{'}(X_s+Y_s)|^2}{(U''(X_s+Y_s))^3} \right. \\
& & \left. + |\theta_s^{\mathcal{H}}|^2 \frac{U^{'}(X_s+Y_s)}{U''(X_s+Y_s)}
+ Z_s \cdot \theta_s^{\mathcal{H}}-\frac12 |Z_s^\cO|^2 \frac{U^{(3)}}{U''}(X_s+Y_s) \right]ds,
\end{array}
\right.
\end{equation}
with the notation $Z=(\underbrace{Z^1,\ldots,Z^{d_1}}_{=:Z^\cH},\underbrace{Z^{d_1+1},\ldots,Z^d}_{=:Z^\cO})$. In addition, the optimal wealth process $X^{\pi^\ast}$ is equal to $X$.
\end{theorem}

\begin{proof}
From Theorem \ref{th:verif1} we know that the optimal strategy is given by
$$ \pi_t^{\ast^i}= -\theta_t^i \frac{U'(X_t^{\pi^\ast}+Y_t)}{U''(X_t^{\pi^\ast}+Y_t)}- Z_t^i, \quad t \in [0,T], \quad i \in \{1,\ldots,d_1\}$$
where $(Y,Z)$ is a solution to the BSDE \eqref{eq:BSDE1} with driver $f$ like in \eqref{eq:driver}. Now plugging the expression of $\pi^\ast$ in relation \eqref{eq:driver} yields
\begin{equation}
\left\lbrace
\begin{array}{lll}
X_t^{\pi^\ast} & = & x-\int_0^t \left(\theta_s \frac{U'(X_s^{\pi^\ast}+Y_s)}{U''(X_s^{\pi^\ast}+Y_s)} + Z_s \right) dW_s^{\mathcal{H}}  -\int_0^t \left(\theta_s \frac{U'(X_s^{\pi^\ast}+Y_s)}{U''(X_s^{\pi^\ast}+Y_s)} + Z_s\right) \cdot \theta_s^{\mathcal{H}} ds\\
\\
Y_t & = & H-\int_t^T Z_s dW_s-\int_t^T \left[-\frac12 |\theta_s^{\mathcal{H}}|^2 \frac{U^{(3)}(X_s^{\pi^\ast}+Y_s) |U^{'}(X_s^{\pi^\ast}+Y_s)|^2}{(U''(X_s^{\pi^\ast}+Y_s))^3} \right. \\
& & \left. + |\theta_s^{\mathcal{H}}|^2 \frac{U^{'}(X_s^{\pi^\ast}+Y_s)}{U''(X_s^{\pi^\ast}+Y_s)} + Z_s \cdot \theta_s^{\mathcal{H}} -\frac12 |Z_s^\cO|^2 \frac{U^{(3)}}{U''}(X_s^{\pi^\ast}+Y_s) \right] ds.
\end{array}
\right.
\end{equation}

Recalling that $X^\pi:=x+ \int_0^\cdot \pi_s (dW_s^{\mathcal{H}}+\theta_s^{\mathcal{H}} ds)$ for any admissible strategy $\pi$, we get the forward part of the FBSDE.
\end{proof}

\begin{remark}
\label{rk:decompU'}
Using It\^o's formula and the FBSDE \eqref{eq:fbsde1}, we have that
$$ U'(X+Y)=U'(x+Y_0) + \int_0^\cdot -\theta_s^\cH U'(X_s+Y_s) dW_s^\cH + \int_0^\cdot U''(X_s+Y_s) Z_s^{\cO} dW_s^{\cO}.$$
\end{remark}

\begin{remark}
\label{re:computprod}
Note that using the system \eqref{eq:fbsde1}, for $\alpha:=U'(X^{\pi^\ast}+Y)$, integration by parts yields for every $t$ in $[0,T]$
\begin{align*}
&U'(X_t^{\pi^\ast}+Y_t) (X_t^{\pi}-X_t^{\pi^\ast})\nonumber\\
&=\int_0^t (X_s^{\pi}-X_s^{\pi^\ast}) d\alpha_s + \int_0^t \alpha_s (\pi_s-\pi_s^{\ast}) dW_s^{\mathcal{H}} \nonumber\\
&\quad + \int_0^t \left( \alpha_s \theta_s^\cH + U''(X_s^{\pi^\ast}+Y_s) (Z_s^\cH+\pi_s^\ast)\right) \cdot (\pi_s-\pi_s^\ast) ds\nonumber\\
&=\int_0^t (X_s^{\pi}-X_s^{\pi^\ast}) d\alpha_s + \int_0^t \alpha_s (\pi_s-\pi_s^{\ast}) dW_s^{\mathcal{H}}
\end{align*}
showing that $U'(X^{\pi^\ast}+Y) (X^{\pi}-X^{\pi^\ast})$ is a local martingale for every $\pi$ in $\Pi^x$.
\end{remark}

The converse implication of Theorems \ref{th:verif1} and \ref{th:verif2} constitutes the second main result.

\begin{theorem}
\label{th:converse}
Let $(H1)$ be satisfied for $U$. Let $(X,Y,Z)$ be a triple of predictable processes which solves the FBSDE \eqref{eq:fbsde1}
satisfying: $Z$ is in $\mathbb{H}^2(\real^d)$, $\E[\vert U(X_T+H) \vert]<\infty$, $\E[|U'(X_T+H)|^2]<\infty$, and $U'(X+Y)$ is a positive martingale.
Moreover, assume that there exists a constant $\kappa>0$ such that $$ -\frac{U'(x)}{U''(x)} \leq \kappa $$ for all $x \in \real$.
Then
$$\pi^{\ast^i}_t:=-\frac{U'(X_t+Y_t)}{U''(X_t+Y_t)} \theta_t^i - Z_t^i, \quad t\in [0,T], \quad i \in \{1,\ldots,d_1\},$$
is an optimal solution of the optimization problem \eqref{eq:opti}.
\end{theorem}

\begin{proof}
Note first that by definition of $\pi^\ast$, $X=X^{\pi^\ast}$.
Since the risk tolerance $ -\frac{U'(x)}{U''(x)}$ is bounded and since $Z$ is in $\mathbb{H}^2(\real^d)$, we immediately get $\E\left[ \int_0^T |\pi^\ast_s|^2 ds \right]<\infty$, thus, $\pi\in\Pi^x$.
By assumption, $U'(X + Y)$ is a positive continuous martingale, hence there exists a continuous local martingale $L$ such that $U'(X+Y)=\mathcal{E}(L)$. And we know from Remark \ref{rk:decompU'} that
$$L= \log(U'(x+Y_0)) + \int_0^\cdot -\theta_s^\cH dW_s^\cH + \int_0^\cdot \frac{U''(X_s+Y_s)}{U'(X_s+Y_s)} Z_s^{\mathcal{O}} dW_s^{\mathcal{O}}.$$
Define the probability measure $\mathbb{Q} \sim \IP$ by $$ \frac{d \mathbb{Q}}{ d \mathbb{P} } := \frac{U'(X_T + H)}{\E[U'(X_T + H)]}.$$
Girsanov's theorem implies that $\tilde{W}:=\tilde{W}^\cH+\tilde{W}^\cO=(W^1 + \theta^1 \cdot dt,\ldots,W^{d_1} + \theta^{d_1} \cdot dt,W^{d_1+1} - \frac{U''(X+Y)}{U'(X+Y)} Z^{d_1+1} \cdot dt,\ldots,W^{d_2} - \frac{U''(X+Y)}{U'(X+Y)} Z^{d_2} \cdot dt)$ is a standard Brownian motion under $\Q$. Thus $X^\pi$ is a local martingale under $\Q$ for every $\pi$ in $\Pi^x$.
Now fix $\pi$ in $\Pi^x$ with $\E[\vert U(X_{T}^\pi+H) \vert]<\infty$. Let $(\tau_n)_n$ be a localizing sequence for the local martingale $X^\pi-X^{\pi^\ast}$.
Since $U$ is a concave, we have
\begin{equation}
\label{eq:reversegene1}
U(X_T^{\pi}+H)-U(X_T^{\pi^\ast}+H) \leq U'(X_T^{\pi^\ast}+H) (X_T^{\pi}-X_T^{\pi^\ast}).
\end{equation}
Taking expectations in \eqref{eq:reversegene1} we get
\begin{align*}
\frac{\E[U(X_T^{\pi}+H)-U(X_T^{\pi^\ast}+H)]}{\E[U'(X_T + H)]} &\leq \E_\Q[X_T^{\pi}-X_T^{\pi^\ast}]\\
& = \E_\Q\left[\lim_{n\to \infty }\int_0^{T \wedge \tau_n} (\pi_s-\pi^\ast_s) d\tilde{W}_s^\cH\right]\\
& = \lim_{n\to \infty }\E_\Q\left[\int_0^{T \wedge \tau_n} (\pi_s-\pi^\ast_s) d\tilde{W}_s^\cH\right]=0,
\end{align*}
which eventually follows as a consequence of Lebesgue's dominated convergence theorem. To this end we prove that
$$ \E_\Q\left[ \sup_{t\in [0,T]} \left|\int_0^t (\pi_s-\pi^\ast_s) d\tilde{W}_s^\cH\right| \right]<\infty. $$
Indeed the BDG inequality and the Cauchy-Schwarz inequality imply that
\begin{align*}
&\E_\Q\left[ \sup_{t\in [0,T]} \left|\int_0^t (\pi_s-\pi^\ast_s) d\tilde{W}_s^\cH\right| \right]\\
& \leq C \E_\Q\left[\left(\int_0^T |\pi_s-\pi^\ast_s|^2 ds\right)^{\frac12} \right]\\
& = C \E\left[\frac{U'(X_T+H)}{\E[U'(X_T+H)]} \left(\int_0^T |\pi_s-\pi^\ast_s|^2 ds\right)^{\frac12} \right]\\
& \leq C \E\left[\left|\frac{U'(X_T+H)}{\E[U'(X_T+H)]}\right|^2\right]^{\frac12} \E\left[ \int_0^T |\pi_s-\pi^\ast_s|^2 ds \right]^\frac12<\infty.
\end{align*}
\end{proof}

We have proved in Theorem \ref{th:verif2} that if \eqref{eq:opti} exhibits an optimal strategy $\pi^\ast\in \Pi^x$, then there exists an adapted solution to the FBSDE \eqref{eq:fbsde1}. As a byproduct we showed the optimization procedure singles out a ``pricing measure'' under which the asset prices and marginal utilities are martingales. In that sense, the process $Y$ captures the impact of future trading gains on the agent's marginal utilities. If we assume additional conditions on the utility function $U$, we get the following regularity properties of the solution $(X,Y,Z)$.

\begin{prop}
\label{prop:regularity}
Assume that for $H\in L^\infty(\Omega,\cF_T,\IP)$ and that the FBSDE \eqref{eq:fbsde1} admits an adapted solution $(X,Y,Z)$ such that $Y$ is bounded. Let
$$\varphi_1(x):=\frac{U'(x)}{U''(x)}, \quad \varphi_2(x):=\frac{U^{(3)}(x) |U'(x)|^2}{(U''(x))^3}, \quad \varphi_3(x):=\frac{U^{(3)}(x)}{U''(x)}, \quad x\in \real. $$
Assume that $U$ is such that $\varphi_i$, $i=1,2,3$ are bounded and Lipschitz continuous functions. Then $(X,Y,Z)$ is the unique solution of \eqref{eq:fbsde1} in $\cS^2(\real) \times \cS^\infty(\real) \times \mathbb{H}^2(\real^d)$. In addition,
$Z\cdot W$ is a BMO-martingale.
\end{prop}

\begin{proof}
Let $(X,Y,Z)$ be a solution to \eqref{eq:fbsde1} such that $Y$ is bounded. Then, using the usual theory on quadratic growth BSDEs (see for example \cite[Theorem 2.5 and Lemma 3.1]{Morlais}) we have only from the backward part of the FBSDE that $Z$ is in $\mathbb{H}^2(\real^d)$ and that $Z \cdot W$ is a BMO-martingale. In addition there exists a unique solution to the backward component in this space for a given process $X$. Now the previous regularity properties of the processes $(Y,Z)$ imply that $X$ is in $\cS^2(\real)$. We turn to the uniqueness of the $X$ process. Assume that there exists another solution $(X',Y',Z')$ of \eqref{eq:fbsde1}. Hence, Theorem \ref{th:converse} implies that $\pi^{\ast'}:=-\frac{U'(X'+Y')}{U''(X'+Y')} \theta^i + Z'^i, \quad i \in \{1,\ldots,d_1\}$ is an optimal solution to our original problem \eqref{eq:opti} and $X'$ is the optimal wealth process. However, by strict concavity of $U$ and by convexity of $\Pi^x$ the optimal strategy has to be unique. So $X$ and $X'$ are the wealth processes of the same optimal strategy, thus, they have to coincide (for instance $X_T=X_T', \, \P-a.s.$) which implies $(Y',Z') = (Y,Z)$.
\end{proof}

In the complete case we are able to construct the solution $(X,Y,Z)$. This is the subject of the following Section. 

\subsection{Characterization and verification: complete markets}
\label{section:completereal}
In this section we consider the benchmark case of a complete market. We assume $d=1$ for simplicity. $H$ denotes a square integrable random variable measurable with respect to the Brownian motion $W$.\\\\
\noindent
In the complete case we can give sufficient conditions for the existence of a solution to the system \eqref{eq:fbsde1}. Our construction relies on the following remark.
\begin{remark}
Using \eqref{eq:fbsde1} the martingale $U'(X^{\pi^\ast}+Y)$ becomes more explicit, because It\^o's formula applied to $U'(X^{\pi^\ast}+Y)$ yields
\begin{align*}
U'(X_t^{\pi^\ast}+Y_t)&=U'(x+Y_0) + \int_0^t U''(X_s^{\pi^\ast}+Y_s) (\pi^\ast_s+Z_s) dW_s \\
&=U'(x+Y_0) - \int_0^t U'(X_s^{\pi^\ast}+Y_s) \theta_s dW_s,
\end{align*}
where we have replaced $\pi^\ast$ by its characterization in terms of $(X,Y,Z)$ from Theorem \ref{th:verif1}. Hence,
\begin{equation}
\label{eq:mart}
U'(X^{\pi^\ast}_t+Y_t)=U'(x+Y_0) \mathcal{E}(-\theta \cdot W)_t, \quad t \in [0,T].
\end{equation}
\end{remark}
This remark will allow us to prove existence of a solution to the system \eqref{eq:fbsde1} under a condition on the risk aversion coefficient $-\frac{U''}{U'}$ of $U$. To this end, we give a sufficient condition on $U$ for the system \eqref{eq:fbsde1} to exhibit a solution. We have the following remark.
\begin{remark}
If $(X,Y,Z)$ is an adapted solution to the system \eqref{eq:fbsde1}, then $P:=X+Y$ is solution of the forward SDE
\begin{equation}
\label{eq:SDE1}
P_t=x+Y_0-\int_0^t \theta_s \frac{U'(P_s)}{U''(P_s)} dW_s - \int_0^t \frac12 |\theta_s|^2 \frac{U^{(3)}(P_s) |U^{'}(P_s)|^2}{(U''(P_s))^3} ds, \quad t\in [0,T].
\end{equation}
In addition, if $(X,Y,Z)$ is in $\cS^2(\real) \times \cS^2(\real) \times \mathbb{H}^2(\real^d)$, then $P \in \cS^2(\real)$. Thus a necessary condition for the FBSDE \eqref{eq:fbsde1} to have a solution is that the SDE \eqref{eq:SDE1} admits a solution.
\end{remark}

We are now going to state an existence result for the FBSDE system \eqref{eq:fbsde1} that characterizes optimal trading strategies in terms of the functions $\varphi_1(x)=\frac{U'(x)}{U''(x)}$ and $\varphi_2(x)=\frac{U^{(3)}(x) |U^{'}(x)|^2}{(U''(x))^3}$.

\begin{prop}
\label{prop:exist1}
Assume that the functions $\varphi_1$ and $\varphi_2$ are bounded and Lipschitz continuous. Then the FBSDE
\begin{equation}
\label{eq:fbsde1complete}
\left\lbrace
\begin{array}{l}
X_t=x-\int_0^t \left( \theta_s \frac{U'(X_s+Y_s)}{U''(X_s+Y_s)} + Z_s\right) dW_s -\int_0^t \left( \theta_s \frac{U'(X_s+Y_s)}{U''(X_s+Y_s)} + Z_s\right) \cdot \theta_s ds\\
\\
Y_t=H-\int_t^T Z_s dW_s -\int_t^T \left( -\frac12 |\theta_s|^2 \frac{U^{(3)}(X_s+Y_s) |U^{'}(X_s+Y_s)|^2}{(U''(X_s+Y_s))^3} + |\theta_s|^2 \frac{U^{'}(X_s+Y_s)}{U''(X_s+Y_s)} \right. \\ \left.  ~~~~~~~~~~~~~~~~~~~~~~~~~~~~~~~~~~~ + Z_s \cdot \theta_s \right) ds
\end{array}
\right.
\end{equation}
admits a solution $(X,Y,Z)$ in $\cS^2(\real) \times \cS^2(\real) \times \mathbb{H}^2(\real^d)$ such that $\E[|U(X_T+H)|]<\infty$ and $\E[|U'(X_T+H)|^2]<\infty$.
\end{prop}

\begin{proof}
Let $m$ in $\real$. Consider the following SDE
$$ P_t^m=x+m-\int_0^t \theta_s \varphi_1(P_s^m) dW_s - \int_0^t \frac12 |\theta_s|^2 \varphi_2(P_s^m) ds, \quad t \in [0,T].$$
Since this SDE has Lipschitz coefficients the existence and uniqueness of a solution in $\cS^2(\real)$ is guaranteed (see for example \cite[V.3. Lemma 1]{Protter}). Next, consider the BSDE
\begin{equation}
\label{eq:Y^m}
Y_t^m=H-\int_t^T Z_s^m dW_s-\int_t^T \left( -\frac12 |\theta_s|^2 \varphi_2(P_s^m) + |\theta_s|^2 \varphi_1(P_s^m) + Z_s^m \cdot \theta_s \right) ds.
\end{equation}
We denote its driver by $f(s,p,z):=-\frac12 |\theta_s|^2 \varphi_2(p) + |\theta_s|^2 \varphi_1(p) + z \cdot \theta_s$. Using the regularity properties of $\varphi_1$ and $\varphi_2$ and the fact that $\theta$ is bounded, there exists a constant $K>0$ such that
$$ |f(s,p,z)| \leq K (1+|z|) $$
and the constant $K$ depends only on $\alpha_1, \alpha_2$ and on $\|\theta\|_{\infty}$, thus in particular $K$ does not depend on $m$. Since the driver $f$ is Lipschitz in $z$, there exists a unique pair of adapted processes $(Y^m,Z^m)$ in $\cS^2(\real) \times \mathbb{H}^2(\real^d)$ which solves \eqref{eq:Y^m}. In addition, $|Y^m_t| \leq K$ holds $\P$-a.s. for all $t$ in $[0,T]$. We recall that this constant $K$ does not depend on $m$, thus $|Y_0^m|\leq K$. Using usual arguments we can show that the map $m \mapsto Y_0^m$ is continuous. Even if this procedure is somehow standard, we reprove this fact here to make the paper self-contained. Fix $m,m'$ in $\real$ with $m\neq m'$. We set $\delta Y_t:=Y_t^m-Y_t^{m'}$, $\delta Z_t:=Z_t^m-Z_t^{m'}$. By \eqref{eq:Y^m} it follows that $(\delta Y, \delta Z)$ is solution to the Lipschitz BSDE:
$$ \delta Y_t =0 -\int_t^T \delta Z_s dW_s -\int_t^T (\theta_s \delta Z_s + h(s)) ds $$
with $h(s):= \frac12 |\theta_s|^2 (\varphi_2(P_s^m)-\varphi_2(P_s^{m'})) + |\theta_s|^2 (\varphi_1(P_s^m)-\varphi_1(P_s^{m'}))$. Using classical a priori estimates for Lipschitz growth BSDEs (see for example \cite[Lemma 2.2]{MaZhang}) we get that:
$$ |\delta Y_0|^2 \leq \E[\sup_{t \in [0,T]} |\delta Y_t|^2 ] \leq C \E\left[\int_0^T |h(s)|^2 ds\right].$$
The boundedness of $\theta$ and the Lipschitz assumption on $\varphi_1$ and on $\varphi_2$ immediately imply that
$$ \E\left[\int_0^T |h(s)|^2 ds\right] \leq C \E\left[\int_0^T |P_s^m - P_s^{m'}|^2 ds\right] \leq C \E\left[\sup_{t\in [0,T]} |P_t^m - P_t^{m'}|^2\right].$$
Combining the inequalities above with classical estimates on Lipschitz SDEs (see for example \cite[Estimate (***) in the proof of Theorem V.7.37]{Protter}) we finally get that
$$ |\delta Y_0|^2 \leq C |m-m'|^2 $$ which concludes the proof by letting $m'$ tending to $m$.
This conjunction with $m \mapsto Y_0^m$ being bounded yields that there exists an element $m^\ast \in \real$ such that $Y_0^{m^\ast}=m^\ast$. Setting
$$ X_t^{m^\ast}:=P_t^{m^\ast}-Y_t^{m^\ast}, \quad t \in [0,T], $$
it is straightforward to check that $(X^{m^\ast},Y^{m^\ast},Z^{m^\ast})$ satisfies  \eqref{eq:fbsde1complete}. Moreover, we have $X^{m^\ast}\in \cS^2(\real)$ since $Y^{m^\ast}$ is bounded and since $P^{m^\ast}\in\cS^2(\real)$. Next, note that $\E[|U'(X_T+Y_T)|^2]<\infty$ since $U'(X_T+Y_T)=U'(x+m)\mathcal{E}(-\theta \cdot W)$. Now using the concavity of $U$, it holds that
$$ U(x) \leq U'(0) x +U(0), \quad -U(x) \leq -U'(x) x-U(0), \quad \forall x \in \real. $$
Consequently,
$$ \E[|U(X_T+H)|] \leq \E[|U'(0)| ~ |X_T+H| + |U(0)|] + \E[|U'(X_T+H) (X_T+H)| + |U(0)|] <\infty. $$
\end{proof}

\section{Utility functions on the positive half-line}
\label{section:halfline}

In this section we study utility functions $U:\real^+ \to \real$ defined on the positive half-line. Again, we assume that $U$ is strictly increasing and strictly concave. \\\\\noindent
In the previous section we have derived a FBSDE characterization of the optimal strategy for the utility maximization problem \eqref{eq:opti}. The key observation was that there exists a stochastic process $Y$ such that $U'(X^{\pi^\ast}+Y)$ is a martingale. However if $U$ is only defined on the positive half-line, it is not clear \textit{a priori} that the expression $U'(X^{\pi^\ast}+Y)$ makes sense. We could generalize this approach by looking for a function $\Phi$ such that $\Phi(X_t^{\pi^\ast},Y_t)$ is a martingale and such that $\Phi(X_T^{\pi^\ast},Y_T)=U'(X_T^{\pi^\ast}+H)$. When $H=0$, it turns out that a good choice of function $\Phi$ is $\Phi(x,y):=U'(x)\exp(y)$ since the system we obtain coincides (up to a non-linear transformation) with the one obtained by Peng in \cite[Section 4]{Peng} using the maximum principle. Note that the system of Peng is not formulated as a FBSDE but rather as a system of equations: one for the wealth process whose dynamics depend on the strategy and one adjoint equation, but a reformulation of this system of equation allows to get a FBSDE (details are given in Section \ref{section:max}).\\\\
\noindent
In the previous section, $\pi$ denoted the total amount of money invested into the stock (the number of shares being $\pi/\tilde{S}$). Now we denote by $\pi^i$ the proportion of wealth invested in the $i$-th stock $S^i$. Once again we denote by $\Pi^x$ the set of admissible strategies with initial capital $x$ which is now defined by
\begin{equation}
\label{eq:admissiblebispower}
\Pi^x:=\L\{ \pi:\Omega \times [0,T] \to \real^{d_1}, \; \pi \mbox{ is predictable}, \;  \E\L[\int_0^T |\pi_s|^2 ds\R]<\infty \R\}.
\end{equation}
The associated wealth process is given by
$$ X_t^\pi:=x+\int_0^t \pi_s X_s^{\pi} dS_s^{\mathcal{H}}, \quad t\in [0,T].$$
Again, we extend $\pi$ to $\real^d$ via $\tilde{\pi}:=(\pi^1,\ldots,\pi^{d_1},0,\ldots,0)$ and make the convention that we write $\pi$ instead of $\tilde{\pi}$.
Thus, we have
$$ X_t^\pi=x \mathcal{E}\left(\int_0^\cdot \pi_r dS_r^{\mathcal{H}} \right)_t, \quad t \in [0,T].$$

From now one we consider a positive $\mathcal{F}_T$-measurable random variable $H$. We furthermore need to impose the following assumptions on $U$.
\\\\\textbf{(H3)} $U:\real^+\to \real$ is three times differentiable, strictly increasing and concave\\\\
\textbf{(H4)} We say that assumption (H4) holds for an element $\pi^\ast$ in $\Pi^x$, if
\begin{itemize}
\item[(i)] $\E[|X_T^{\pi^\ast} U'(X_T^{\pi^\ast}+H)|^2]<\infty;$
\item[(ii)]  the sequence of random variables
$$\left( \frac{1}{\varepsilon}(X_T^{\pi^\ast+\varepsilon \rho}-X_T^{\pi^\ast}) \int_0^1 U'(X_T^{\pi^\ast}+H+r (X_T^{\pi^\ast+\varepsilon \rho}-X_T^{\pi^\ast})) dr \right)_{\varepsilon \in (0,1)}$$
is uniformly integrable;
\item[(iii)]  $$ \lim_{\varepsilon \to 0} \sup_{t\in [0,T]} \E\left[\left|\frac{1}{\varepsilon}(X_t^{\pi^\ast+\varepsilon \rho}-X_t^{\pi^\ast})-\xi_t\right|^2\right] = 0, $$ 
where $d\xi_t=\pi_t^\ast \xi_t dS_t^\cH + \rho_t X_t^{\pi^\ast} dS_t^\cH, \quad t\in [0,T]$, and $\sup_{t\in[0,T]} \E[\vert \xi_t \vert^2]<\infty$.
\end{itemize}
\textbf{(H5)} There exists a constant $c>0$ such that $\frac{-U'(x)}{x U''(x)}\leq c$ for all $x \in \real^+$.\\\\

\subsection{Characterization and verification: incomplete markets}

Note that in condition (H4), if $U'(0)<\infty$ or if $H\geq a>0$ is satisfied, then (iii) implies (ii).

\begin{theorem}\label{theorem:positive_opt_bsde}
Assume that (H3) holds and that $H$ is a positive random variable belonging to $L^2(\Omega,\mathcal{F}_T,\P)$. Let $\pi^\ast$ be an optimal solution to \eqref{eq:opti} satisfying $\E[|U(X_T^{\pi^\ast}+H)|]<\infty$ and which satisfies assumption (H4). Then there exists a predictable process $Y$ with $Y_T = \log(U'(X_T^{\pi^\ast}+H))-\log(U'(X_T^{\pi^\ast}))$ such that $X^{\pi^\ast}U'(X^{\pi^\ast}) \exp(Y)$ is a martingale and $$\pi_s^{\ast^i}=- \frac{U'(X_s^{\pi^\ast})}{X_s^{\pi^\ast} U''(X_s^{\pi^\ast})} (Z_s^i+\theta_s^i), \quad s \in [0,T], \quad i=1,\ldots,d_1,$$ where $Z_t:=\left(\frac{d\langle Y, W^1 \rangle_t}{dt},\ldots,\frac{d\langle Y, W^d \rangle_t}{dt}\right)$.
\end{theorem}

\begin{proof}
As in the proof of Theorem \ref{th:verif1}, we prove the existence of $Y$ such that $X^{\pi^\ast} U'(X^{\pi^\ast})\exp(Y)$ is a martingale with $Y_T=\log(U'(X_T^{\pi^\ast}+H))-\log(U'(X_T^{\pi^\ast}))$. Consequently, $U'(X_T^{\pi^\ast}+H)=U'(X_T^{\pi^\ast})\exp(Y_T)$. By (H4), the process
\begin{align*}
\alpha_t&:=\E[X_T^{\pi^\ast} U'(X_T^{\pi^\ast}+H) \vert \mathcal{F}_t]
\end{align*}
is a square integrable martingale. In addition it is the unique solution to the BSDE
$$ \alpha_t=X_T^{\pi^\ast} U'(X_T^{\pi^\ast}+H) -\int_t^T \beta_s dW_s, \quad t \in [0,T],$$
where $\beta$ is a square integrable predictable process with values in $\real^d$. We set $Y:=\log(\alpha)-\log(U'(X^{\pi^\ast}))-\log(X^{\pi^\ast})$. As in the proof of Theorem \ref{th:verif1}, It\^o's formula implies that
\begin{align*}
Y_t&=Y_T-\int_t^T \left[\frac{\beta_s}{\alpha_s}- \frac{U''(X_s^{\pi^\ast})}{U'(X_s^{\pi^\ast})} X_s^{\pi^\ast} \pi_s^\ast- \pi_s^\ast \right] dW_s \\
& \quad \quad  - \int_t^T \bigg[-\frac12 \frac{|\beta_s|^2}{|\alpha_s|^2} -\left(\frac{U''(X_s^{\pi^\ast})}{U'(X_s^{\pi^\ast})} X_s^{\pi^\ast} \pi_s^\ast + \pi_s^\ast\right) \cdot \theta_s^{\mathcal{H}}\\
&\quad \quad \quad + \frac{|X_s^{\pi^\ast} \pi_s^\ast|^2}{2} \left( \left\vert\frac{U''(X_s^{\pi^\ast})}{U'(X_s^{\pi^\ast})}\right\vert^2 - \frac{U^{(3)}(X_s^{\pi^\ast})}{U'(X_s^{\pi^\ast})} \right) + \frac{|\pi_s^\ast|^2}{2} \bigg] ds.
\end{align*}
Setting
\begin{equation}
\label{eq:relabetqZ}
Z_t^i=\frac{\beta_t^i}{\alpha_t}-\frac{\pi_t^\ast}{U'(X_t^{\pi^\ast})} (X_t^{\pi^\ast} U''(X_t^{\pi^\ast})+U'(X_t^{\pi^\ast})), \quad i=1,\ldots,d,
\end{equation}
we get that
\begin{eqnarray*}
Y_t&=&Y_T-\int_t^T Z_s dW_s -\int_t^T \bigg[-\frac12 \frac{U^{(3)}(X_s^{\pi^\ast})}{U'(X_s^{\pi^\ast})} |X_s^{\pi^\ast} \pi_s^\ast|^2 - (Z_s^\cH+\theta_s^\cH) \cdot \left( \frac{U''(X_s^{\pi^\ast})}{U'(X_s^{\pi^\ast})} X_s^{\pi^\ast} \pi_s^\ast + \pi_s^\ast \right) \\
&& \quad \quad \quad - \frac{U''(X_s^{\pi^\ast})}{U'(X_s^{\pi^\ast})} X_s^{\pi^\ast} |\pi_s^\ast|^2 - \frac12 |Z_s|^2 \bigg] ds, \quad t\in [0,T].\end{eqnarray*}
We now derive the characterization of $\pi^\ast$ in terms of $U'$ and $Y$ and $Z$. We employ an argument put forth in \cite{Peng} and then substitute the Hamiltonian by a BSDE. Fix $\pi \in \Pi^x$. Since the latter is a convex set, for $\rho:=\pi-\pi^\ast$, the $\pi^\ast+\varepsilon \rho$ is an admissible strategy for every $\varepsilon \in (0,1)$. We have
\begin{align*}
&\frac{1}{\varepsilon}(U(X_T^{\pi^\ast+\varepsilon \rho}+H)-U(X_T^{\pi^\ast}+H)) = \\
&\hspace{2cm}\frac{1}{\varepsilon}(X_T^{\pi^\ast+\varepsilon \rho}-X_T^{\pi^\ast}) \int_0^1 U'(X_T^{\pi^\ast}+H+r (X_T^{\pi^\ast+\varepsilon \rho}-X_T^{\pi^\ast})) dr.
\end{align*}
Since $\pi^\ast$ is optimal we find
\begin{equation}
\label{eq:varhl1}
\E\left[\frac{1}{\varepsilon}(X_T^{\pi^\ast+\varepsilon \rho}-X_T^{\pi^\ast}) \int_0^1 U'(X_T^{\pi^\ast}+H+r (X_T^{\pi^\ast+\varepsilon \rho}-X_T^{\pi^\ast})) dr\right]\leq 0, \quad \forall \varepsilon>0.
\end{equation}
Now let $\xi$ be defined as
$$ d\xi_t=(\pi_t^\ast \xi_t + \rho_t X_t^{\pi^\ast}) dS_t^\cH, \quad t\in [0,T].$$
By (H4), we can apply Lebesgue's dominated convergence theorem in inequality \eqref{eq:varhl1} which, possibly passing to a subsequence, yields
\begin{align*}
&\E[ \xi_T U'(X_T^{\pi^\ast}+H) ] = \lim_{\varepsilon \to 0} \E\left[\frac{1}{\varepsilon}(X_T^{\pi^\ast+\varepsilon \rho}-X_T^{\pi^\ast}) \int_0^1 U'(X_T^{\pi^\ast}+H+r (X_T^{\pi^\ast+\varepsilon \rho}-X_T^{\pi^\ast})) dr \right].
\end{align*}
Combined with \eqref{eq:varhl1}, it leads to
\begin{equation}
\label{eq:temp1}
\E[ \xi_T (X_T^{\pi^\ast})^{-1} U'(X_T^{\pi^\ast}) X_T^{\pi^\ast} \exp(Y_T) ] = \E[ \xi_T U'(X_T^{\pi^\ast}+H) ] \leq 0, \quad \forall \pi \in \Pi^x.
\end{equation}
We now restrict consideration to a particular class of processes $\pi$, that is, we choose $\rho$ to be a bounded predictable process and we define $\pi:=\rho+\pi^\ast$ which is admissible strategy since it is square integrable. The integration by parts formula for continuous semimartingales implies that
$$ \xi_t (X_t^{\pi^\ast})^{-1}= \int_0^t \rho_s dW_s^\cH + \int_0^t [\rho_s \cdot \theta_s^\cH - \rho_s \cdot \pi_s^{\ast}] ds, \quad t \in [0,T].$$
Another application of integration by parts to $\alpha=U'(X^{\pi^\ast}) X^{\pi^\ast} \exp(Y)$ and $\xi (X^{\pi^\ast})^{-1}$ yields
\begin{align}
\label{eq:temp2}
\xi_T U'(X_T^{\pi^\ast}+Y_T)&=\xi_T (X_T^{\pi^\ast})^{-1} U'(X_T^{\pi^\ast}) X_T^{\pi^\ast} \exp(Y_T) \nonumber\\
&=\int_0^T \xi_t (X_t^{\pi^\ast})^{-1} d\alpha_t + \int_0^T \alpha_t \rho_t dW_t^{\mathcal{H}}\nonumber\\
&\quad + \int_0^T \rho_t \exp(Y_t) X_t^{\pi^\ast} \cdot (U'(X_t^{\pi^\ast}) (Z_t^\cH+\theta_t^\cH) + U''(X_t^{\pi^\ast}) X_t^{\pi^\ast} \pi_t^\ast) dt.
\end{align}
We now intend to take the expectation in the above relation. To this end, we need the following moment estimates. Using that $\rho$ is bounded, we have
\begin{align}
\label{eq:est1}
\E[\sup_{t\in [0,T]}|\xi_t (X_t^{\pi^\ast})^{-1}|^2]&=\E\left[\sup_{t\in [0,T]}\left\vert \int_0^t \rho_s dW_s^\cH +\int_0^t (\rho_s \cdot \theta_s^\cH -\rho_s \cdot \pi_s^\ast) ds \right\vert^2\right]\nonumber\\
&\leq C \E\left[\sup_{t\in [0,T]}\left\vert \int_0^t \rho_s dW_s^\cH \R|^2\R] + \E\left[\sup_{t\in [0,T]}\L| \int_0^t |\rho_s \cdot \theta_s^\cH -\rho_s \cdot \pi_s^\ast| ds \right\vert^2\right]\nonumber\\
&\leq C \L(\E\left[\int_0^T |\rho_s|^2 ds \right] + \E\left[\L|\int_0^T \rho_s \cdot \theta_s^\cH ds\R|^2\R] + \E\left[\L|\int_0^T \rho_s \cdot \pi_s^\ast ds\R|^2\R] \R)\nonumber\\
&\leq C\L(1+\E\L[\int_0^T |\pi^\ast_s|^2 ds \R]\R)<\infty,
\end{align}
where we have used Doob's inequality. Consequently, we get
$$\E[|\xi_T (X_T^{\pi^\ast})^{-1} \alpha_T|]\leq \E[|\alpha_T|^2]^{1/2} \E[||\xi_T (X_T^{\pi^\ast})^{-1}|^2]^{1/2} <\infty,$$
which follows from the Cauchy-Schwarz inequality. With $\rho$ being bounded, we get for some generic constant $C>0$
$$ \E\L[\int_0^T |\alpha_s \rho_s|^2 ds\R]\leq C \E\L[ \int_0^T |\alpha_s|^2 ds\R]<\infty. $$
Hence $\int_0^\cdot \alpha_t \rho_t dW_t^{\mathcal{H}}$ is a square integrable martingale.
Next, let $(\tau_n)_{n\geq 1}$ be a localizing sequence for the local martingale $ \int_0^\cdot \xi_t (X_t^{\pi^\ast})^{-1} d\alpha_t$. Then we have
$$ \left\vert \int_0^{\tau_n} \xi_t (X_t^{\pi^\ast})^{-1} d\alpha_t \right\vert \leq \sup_{t\in [0,T]} \left\vert \int_0^t \xi_t (X_t^{\pi^\ast})^{-1} d\alpha_t \right\vert.$$
To apply Lebesgue's dominated convergence theorem and show that $\E\L[\int_0^T \xi_t (X_t^{\pi^\ast})^{-1} d\alpha_t\R]=0$, we need to prove $\E\L[\sup_{t\in [0,T]} \left\vert \int_0^t \xi_t (X_t^{\pi^\ast})^{-1} d\alpha_t \right\vert\R]<\infty$:
\begin{align*}
\E\left[\sup_{t\in [0,T]} \left|\int_0^t \xi_t (X_t^{\pi^\ast})^{-1} d\alpha_t \right|\right]
&\leq C \E\left[\left|\int_0^T |\xi_t|^2 |(X_t^{\pi^\ast})^{-1}|^2 d\langle \alpha \rangle_t \right|^{1/2}\right]\\
&\leq C \E\left[\sup_{t\in [0,T]} |\xi_t|^2 |(X_t^{\pi^\ast})^{-1}|^2\right]^{1/2} \E\left[\langle \alpha \rangle_T \right]^{1/2}\\
& <\infty,
\end{align*}
where we have used the estimate \eqref{eq:est1}. Thus, by \eqref{eq:temp2} it follows that
$$ \E\left[\left| \int_0^T \rho_t \exp(Y_t) X_t^{\pi^\ast} \cdot (U'(X_t^{\pi^\ast}) (Z_t^\cH+\theta_t^\cH) + U''(X_t^{\pi^\ast}) X_t^{\pi^\ast} \pi_t^\ast) dt \right|\right]<\infty,$$
and from \eqref{eq:temp1}, it holds that for every $\pi$ in $\Pi^x$ such that $\rho$ is bounded, we get
$$ \E\left[\int_0^T \rho_t \exp(Y_t) X_t^{\pi^\ast} \cdot (U'(X_t^{\pi^\ast}) (Z_t^\cH+\theta_t^\cH) + U''(X_t^{\pi^\ast}) X_t^{\pi^\ast} \pi_t^\ast) dt \right] \leq 0. $$
Substituting $\rho$ with $-\rho$ in the previous inequality, we obtain for every $\rho$
\begin{equation}
\label{eq:finalpower}
\E\left[\int_0^T \rho_t \exp(Y_t) X_t^{\pi^\ast} \cdot (U'(X_t^{\pi^\ast}) (Z_t^\cH+\theta_t^\cH) + U''(X_t^{\pi^\ast}) X_t^{\pi^\ast} \pi_t^\ast) dt \right] = 0.
\end{equation}
Now let $A_t:=U'(X_t^{\pi^\ast}) (Z_t^\cH+\theta_t^\cH) + U''(X_t^{\pi^\ast}) X_t^{\pi^\ast} \pi_t^\ast$ and let $\rho_t(\omega):=\textbf{1}_{A_t(\omega)>0}$. Recall that we have $d\P\otimes dt$-a.s. $\exp(Y_t) X_t^{\pi^\ast}>0$. Plugging $\rho$ into \eqref{eq:finalpower} yields
$$ A_t(\omega) \leq 0, \; d\P\otimes dt-a.e. $$
Similarly choosing $\rho_t(\omega):=\textbf{1}_{A_t(\omega)<0}$, we find
$$ A_t(\omega) = 0, \; d\P\otimes dt-a.e. $$
Thus, we achieve
$$\pi_t^{\ast^i}=- \frac{U'(X_t^{\pi^\ast})}{X_t^{\pi^\ast} U''(X_t^{\pi^\ast})} (Z_t^i+\theta_t^i), \quad \forall t\in [0,T], \quad i=1,\ldots,d_1.$$
\end{proof}

Let us now deal with converse implication. 
\begin{theorem}
\label{th:conversehalfline}
Assume (H3) and (H5). Let $(X,Y,Z)$ be an adapted solution of the FBSDE

\begin{equation}
\label{eq:fbsde3}
\left\lbrace
\begin{array}{l}
X_t=x-\int_0^t \frac{U'(X_s)}{U''(X_s)} (Z_s^\cH+\theta_s^\cH) dW_s^\cH -\int_0^t \frac{U'(X_s)}{U''(X_s)} (Z_s^\cH+\theta_s^\cH) \theta_s ds,\\\\
Y_t=\log\left(\frac{U'(X_T+H)}{U'(X_T)}\right) - \int_t^T \left[ ( |Z_s^\cH+\theta_s^\cH|^2) \left( 1- \frac12 \frac{U^{(3)}(X_s) U'(X_s)}{|U''(X_s)|^2}\right) - \frac12 |Z_s|^2 \right] ds\\
\hspace{3.55cm}-\int_t^T Z_s dW_s
\end{array}
\right.
\end{equation}
such that $\E[|U(X_T^{\pi^\ast}+H)|]<\infty$, $Z$ is an element of $\mathbb{H}^2(\real^d)$ and the positive local martingale $XU'(X) \exp(Y)$ is a true martingale.
$$\pi^{\ast^i}_t:=-\frac{U'(X_s)}{X_s U''(X_s)} (Z_s^i+\theta_s^i), \quad s\in [0,T], \quad i=1,\ldots,d_1$$
is an optimal solution to the optimization problem \eqref{eq:opti}.
\end{theorem}

\begin{proof}
We first note that $\pi^\ast \in \Pi^x$ since by the fact that $Z$ is in $\mathbb{H}^2(\real^d)$,
there is a constant $C>0$ such that
$$ \E\left[\int_0^T |\pi_t^\ast|^2 dt\right] \leq C~ \E\left[\int_0^T |Z_t^\cH+\theta_t^\cH|^2 dt\right]<\infty.$$
Now let $\pi$ be an element of $\Pi^x$. Let $D:=U'(X) \exp(Y)$.
Applying It\^o's formula and plugging in the expression of $\pi^\ast$, we find that
$$ dD_t = D_t (-\theta_t dW_t^\cH + Z_t dW_t^\cO), \quad D_0=U'(x) \exp(Y_0),$$
hence,
\begin{equation}
\label{eq:dual}
D_t = U'(x) \exp(Y_0) \mathcal{E}\left(-\int_0^\cdot \theta_s dW_s^\cH + \int_0^\cdot Z_s dW_s^{\cO}\right)_t, \quad t\in [0,T],
\end{equation}
which is a positive local martingale.
Now fix $\pi$ in $\Pi^x$. By definition of $X^\pi$ and of $D$, the product formula implies that $X^\pi D$ satisfies
$$ D X^{\pi} = x D_0 \cE((\pi-\theta) \cdot W^\cH + Z \cdot W^\cO). $$
Hence, $X^\pi D$ is a supermartingale and so $\E[D_T X_T^\pi] \leq D_0 x$.
By assumption, $X^{\pi^\ast} D=X U'(X) \exp(Y)$ is a true martingale so $\E[D_T X_T^{\pi^\ast}]=D_0 x$.
Finally combining the facts above, recalling that $D_T=U'(X_T^{\pi^\ast}+H)$ and using the concavity of $U$, 
we obtain
\begin{align}
\label{eq:reverproof2bis}
\E[U(X^{\pi}_T+H)-U(X^{\pi^\ast}_T+H)] &\leq \E[U'(X^{\pi^\ast}_T+H) (X^{\pi}_T-X^{\pi^\ast}_T)]\leq 0.
\end{align}
\end{proof}

\begin{remark}
In the previous proof, if we apply integration by parts formula to $D=U'(X) \exp(Y)$ and $X^\pi-X^{\pi^\ast}$, we get
\begin{equation*}
U'(X^{\pi})\exp(Y) (X^{\pi}-X^{\pi^\ast})=\int_0^\cdot (X_t^\pi-X_t^{\pi^\ast}) dD_t + \int_0^\cdot D_t (\pi_t X_t^\pi-\pi_t^\ast X_t^{\pi^\ast}) dW_t^{\cal{H}},
\end{equation*}
thus $U'(X^{\pi})\exp(Y) (X^{\pi}-X^{\pi^\ast})$ is a local martingale for every admissible strategy $\pi$.
\end{remark}

\begin{remark}
\label{rk:martgamma}
Note that using the regularity assumptions of the FBSDE \eqref{eq:fbsde3}, we derived that $D:=U'(X^{\pi^\ast}) \exp(Y)$ is a true martingale
$$ D_t = U'(x) \exp(Y_0) \mathcal{E}\left(-\theta \cdot W^\cH + Z^\cO \cdot W^\cO \right).$$
\end{remark}

\subsection{Characterization and verification: complete markets}

We adopt the setting and notations of Section \ref{section:halfline} with $d_1=d=1$ and $H=0$. In the complete case we can give sufficient conditions for the existence of a solution to the system \eqref{eq:fbsde3}. To this end, note the following remark.
\begin{remark}
Similar to Remark \ref{rk:martgamma}, we can use \eqref{eq:fbsde3} to characterize further the martingale $U'(X^{\pi^\ast})\exp(Y)$: applying It\^o's formula to $U'(X^{\pi^\ast}) \exp(Y)$ gives rise to
$$ U'(X_t^{\pi^\ast}) \exp(Y_t)=U'(x)\exp(Y_0) - \int_0^t U'(X_s) \exp(Y_s) \theta_s dW_s, $$
hence, we have
\begin{equation}
\label{eq:martcomp}
U'(X^{\pi^\ast}_t)\exp(Y_t)=U'(x) \exp(Y_0) \mathcal{E}(-\theta \cdot W)_t, \quad t \in [0,T].
\end{equation}
\end{remark}
This observation allows to prove the existence of \eqref{eq:fbsde3} under a condition on the risk aversion coefficient $-\frac{U''}{U'}$. Let $\varphi_1(x):=\frac{U'(x)}{U''(x)}$ and $\varphi_2(x):=1-\frac12 \frac{U^{(3)}(x) U'(x)}{|U''(x)|^2}$. We will now give sufficient condition for the system \eqref{eq:fbsde3} to exhibit a solution. We begin with the following remark.
\begin{remark}
Note that if $\varphi_2$ is constant then the system above decouples. An elementary analysis shows that this happens if and only is $U$ is the exponential, power, log or quadratic (mean-variance hedging) function. If $U(x)=-\exp(-\alpha_1 x)-\exp(-\alpha_2 x)$ then $\varphi_2$ is bounded and Lipschitz and if $U(x):=\frac{x^{\gamma_1}}{\gamma_1}+\frac{x^{\gamma_2}}{\gamma_2}$ then $\varphi_2$ is a bounded function.
\end{remark}
\begin{theorem}
Assume that $\varphi_2$ is a continuous bounded function. Then there exists an adapted solution $(X,Y,Z)$ in $\cS^2(\real^{d_1}) \times \cS^2(\real) \times \mathbb{H}^2(\real^d) $ to the FBSDE
\begin{equation}
\label{eq:fbsde3bis}
\left\lbrace
\begin{array}{l}
X_t=x-\int_0^t \frac{U'(X_s)}{U''(X_s)} (Z_s+\theta_s) dW_s -\int_0^t \frac{U'(X_s)}{U''(X_s)} (Z_s+\theta_s) \theta_s ds\\
\\
Y_t=0-\int_t^T Z_s dW_s-\int_t^T \left[ |Z_s+\theta_s|^2 \left( 1- \frac12 \frac{U^{(3)}(X_s) U'(X_s)}{|U''(X_s)|^2}\right) - \frac12 |Z_s|^2 \right] ds.
\end{array}
\right.
\end{equation}
Moreover, $\E[|U(X_T)|]<\infty$ and $\E[|U'(X_T)|^2]<\infty$.
\end{theorem}
\begin{proof}
Fix $m>0$ and consider the BSDE
\begin{align*}
Y_t^m &= 0 - \int_t^T \left[ |Z_s^m+\theta_s|^2 \varphi_2\left((U')^{-1}(U'(x) \exp(m) \mathcal{E}(-\theta \cdot W)_t \exp(-Y_t^m))\right) - \frac12 |Z_s^m|^2 ds \right]\\
&\qquad -\int_t^T Z_s^m dW_s.
\end{align*}
Since $\varphi_2$ is bounded, the driver of the BSDE above in $(Y^m,Z^m)$ can be bounded uniformly in $m$, hence \cite{Kobylanski} yields a pair $(Y^m,Z^m)\in \cS^2(\real) \times \mathbb{H}^2(\real)$ solution to this equation with $|Y^m| \leq C$ where $C$ does not depend on $m$ and $Z \cdot W$ is a BMO-martingale. In addition (once again using standard arguments like in the proof of Proposition \ref{prop:exist1}) we have that $m \mapsto Y_0^m$ is continuous. Thus there exists an element $m^\ast>0$ such that $Y_0^{m^\ast}=m^\ast$. Now applying It\^o's formula to
$$ X^{m^\ast}:=(U')^{-1}(U'(x) \exp(m^\ast) \mathcal{E}(-\theta \cdot W) \exp(-Y^m)),$$
we check that $(X^{m^\ast},Y^{m^\ast},Z^{m^\ast})$ satisfies \eqref{eq:fbsde3bis}. It remains to show that $\E[|U(X_T)|]<\infty$. From the concavity of $U$ we have that
$$ \E[|U(X_T)|] \leq |U'(0)| \E[|X_T|] + |U(0)| + \E[|U'(X_T) X_T|] + |U(0)|.$$
Since $X=x\mathcal{E}(-\frac{U'(X)}{X U''(X)} (Z+\theta)\cdot W)$, $-\frac{U'(x)}{xU''(x)} \leq \kappa$ for $x\in\real$ and $(Z+\theta) \cdot W$ is a BMO-martingale, $X$ is a true martingale, and thus $\E[X_T]=x$. Similarly we have that $X_T U'(X_T)=X_T U'(X_T) \exp(Y_T)=x U'(x) \exp(Y_0) \mathcal{E}((-\frac{U'(X)}{X U''(X)} (Z+\theta)-\theta)\cdot W)$ and so $X U'(X) \exp(Y)$ is a true martingale. This hence proves $\E[|X_T U'(X_T)|]<\infty$.
\end{proof}

\section{Links to other approaches}\label{section:link}

In this section we link our approach to characterizing optimal investment strategies to two other approaches based on the stochastic maximum principle and duality theory, respectively.

\subsection{Stochastic maximum principle}

\label{section:max}

This section links our approach in the complete market setting to the approach using the stochastic maximum principle. As we are interested only in the link, we will only give a formal derivation. In particular, we suppose here that $U$ and $U^{-1}$ are smooth enough with bounded derivatives. Let us consider the complete market case with $d_1=d=1$ for simplicity and $H=0$ and recall that in this setting, the wealth process is given by

$$ X_t^\pi=x+\int_0^t \pi_s dW_s+ \int_0^t \pi_s \theta_s ds, \quad t \in [0,T].$$
We consider $J(\pi):=\E[U(X_T^\pi)]$ and set $\tilde{X}_t^\pi:=U(X_t^\pi)$. It\^o's formula yields
$$ d\tilde{X}^\pi_t=U'(U^{-1}(\tilde{X}^\pi_t)) \pi_t dW_t + \Big[U'(U^{-1}(\tilde{X}^\pi_t)) \pi_t \theta_t +  \frac12 U''(U^{-1}(\tilde{X}^\pi_t)) |\pi_t|^2 \Big] dt $$
and $J(\pi)=\E[\tilde{X}^\pi_T]$. Applying the maximum principle technique described in \cite{Bismut78} (see also \cite[Section 4]{Peng}), we introduce the adjoint equation to get
\begin{equation}
\label{eq:Peng1}
\left\lbrace
\begin{array}{l}
d\tilde{X}^\pi_t=U'(U^{-1}(\tilde{X}^\pi_t)) \pi_t dW_t + \Big[ U'(U^{-1}(\tilde{X}^\pi_t)) \pi_t \theta_t +  \frac12 U''(U^{-1}(\tilde{X}^\pi_t)) |\pi_t|^2 \Big] dt, \; \tilde{X}^\pi_0=U(x),\\
-dp_t= \Big[ \left( \frac{U''}{U'}(U^{-1}(\tilde{X}^\pi_t)) \theta_t \pi_t + \frac12 \frac{U^{(3)}}{U''}(U^{-1}(\tilde{X}^\pi_t)) |\pi_t|^2 \right) p_t + k_t \frac{U''}{U'}(U^{-1}(\tilde{X}^\pi_t)) \pi_t \Big] dt +  k_t dW_t, \; p_T=1.
\end{array}
\right.
\end{equation}
We now introduce the corresponding Hamiltonian, defined as
$$ H(t,p,k,\pi):=p[U'(U^{-1}(\tilde{X}^\pi_t)) \pi_t \theta_t + \frac12 U''(U^{-1}(\tilde{X}^\pi_t)) |\pi_t|^2] + k U'(U^{-1}(\tilde{X}^\pi_t)) \pi_t.$$ A formal maximization gives
$$ \pi^\ast_t:=-\frac{U'}{U''}(U^{-1}(\tilde{X}^\pi_t))\left[\frac{k_t}{p_t}+\theta_t\right].$$
Plugging this into \eqref{eq:Peng1} yields
\begin{equation}
\label{eq:Peng2}
\left\lbrace
\begin{array}{l}
d\tilde{X}^\pi_t=-\frac{|U'|^2}{U''}(U^{-1}(\tilde{X}^\pi_t))\left(\frac{k_t}{p_t}+\theta_t\right) \left[dW_t-\frac12 \left(\frac{k_t}{p_t}-\theta_t\right) dt\right], \; \tilde{X}^\pi_0=U(x),\\
dp_t= -\left(\frac{k_t}{p_t}+\theta_t\right)^2 p_t \left[-1+ \frac12 \frac{U^{(3)} U'}{|U''|^2}(U^{-1}(\tilde{X}^\pi_t))\right] dt + k_t dW_t, \; p_T=1
\end{array}
\right.
\end{equation}
We now relate this system with \eqref{eq:fbsde3bis} using a Cole-Hopf type transformation. First we plug $\pi^\ast$ into \eqref{eq:Peng2} and obtain
\begin{equation}
\label{eq:Peng3}
\left\lbrace
\begin{array}{l}
dX^{\pi^\ast}_t=-\frac{U'}{U''}(X^{\pi^\ast}_t)\left[\frac{k_t}{p_t}+\theta_t\right] (dW_t+\theta dt), \; X^{\pi^\ast}_0=x,\\
dp_t= -\left(\frac{k_t}{p_t}+\theta_t\right)^2 p_t \left[-1+ \frac12 \frac{U^{(3) U'}}{|U''|^2}(X_t^{\pi^\ast})\right] dt + k_t dW_t, \; p_T=1.
\end{array}
\right.
\end{equation}
Next consider the system
\begin{equation}
\label{eq:Peng4}
\left\lbrace
\begin{array}{l}
dX^{\pi^\ast}_t=-\frac{U'}{U''}(X^{\pi^\ast}_t) (Z_t+\theta_t) (dW_t+\theta dt), \; X^{\pi^\ast}_0=x,\\
dY_t= \left[ (Z_t+\theta_t)^2 (1-\frac12 \frac{U^{(3)}(X_t^{\pi^\ast}) U'(X_t^{\pi^\ast})}{|U''|^2(X_t^{\pi^\ast})}) -\frac12 |Z_t|^2 \right] dt + Z_t dW_t, \; Y_T=0.
\end{array}
\right.
\end{equation}
Setting $\tilde{p}_T:=\exp(Y_t)$, $\tilde{k}:=Z \tilde{p}$ and $\tilde{X}:=X$, It\^o's formula implies that $(\tilde{p},\tilde{k})$ is a solution to \eqref{eq:Peng3}.


\subsection{BSDE solution via convex duality methods}
\label{section:duality}

Let us now turn to a very important link of our approach with the convex duality theory.
We have seen in Sections \ref{section:realline} and \ref{section:halfline} that our approach relies on choosing a process $Y$ such that the quantities $U'(X^{\pi^\ast}+Y)$ and $X^{\pi^\ast} U'(X^{\pi^\ast}) \exp(Y)$, respectively, are martingales. In fact, these martingales are not any martingales. For instance in case of a utility function on the whole real line, $U'(X^{\pi^\ast}+Y)$ is exactly $U'(x+Y_0) \mathcal{E}(-\theta \cdot W^\cH + \frac{U''}{U'}(X^{\pi^\ast} +Y) Z^\cO \cdot W^\cO)$. So in the complete case it is exactly the martingale under which the price is itself a martingale. For utility functions defined on the positive half line this leads directly to duality theory, since it is known from the original paper by Kramkov and Schachermayer (\cite{KramkovSchachermayer}) that (under some growth-type condition on $U$) the optimal wealth process $X^{\pi^\ast}$ and the stochastic process $Y^\ast$ solution to the so-called dual-problem are such that the stochastic process $X^{\pi^\ast} Y^\ast$ is a martingale. In addition, with our notations, Kramkov and Schachermayer prove that $Y^\ast$ has the form $Y^\ast= Y^\ast_0 \mathcal{E}(-\theta \cdot W^\cH + M)$ where $M$ is a martingale orthogonal to $W^\cH$. Recall that in our case $X^{\pi^\ast} U'(X^{\pi^\ast}) \exp(Y)$ is a martingale and from \eqref{eq:dual}, we have proved that $D:=U'(X^{\pi^\ast}) \exp(Y)$ is exactly of the form $D_0 \mathcal{E}(-\theta \cdot W^\cH + Z^\cO \cdot W^\cO)$, in other words $Y^\ast=D$ and the $Z^\cO$ component appearing in the solution of our FBSDE exactly represents the orthogonal part in the dual optimizer of Kramkov and Schachermayer theory. Obviously, this needs to be derived more formally. This is the goal of this section.\\\\
\noindent
The aim of this section is to derive a solution of the forward-backward equation \eqref{eq:fbsde3bis} by means of the results from the convex duality approach to \eqref{eq:opti}. We denote by $\Pi^1$ the set of admissible strategies with initial capital one unit of currency. In the case of zero endowment $H=0$, the solution to the concave optimization problem \eqref{eq:opti} is achieved by formulating and solving the following dual problem: denoting the convex conjugate of the concave function $U$ by
\begin{align*}
V(y) := \sup_{x>0} \big\{ U(x) - xy \big\}, \quad y>0,
\end{align*}
where $d X^\pi_t = X^\pi_t \pi_t \frac{d \tilde{S}_t}{\tilde{S}_t}, ~ X^\pi_0 = x > 0,$
and defining a family of nonnegative semimartingales via
\begin{align*}
\cY := \big\{ Y \geq 0: Y_0 = 1, \text{ $X^\pi Y$ is a supermartingale for every }\pi \in \Pi^1 \big\},
\end{align*}
the primal problem \eqref{eq:opti} is solved by solving instead the dual convex optimization problem
\begin{align}
\label{eq:dual_problem1}
v(y) = \inf_{Y_T \in \cY} \E \big[ V(yY_T) \big], \quad y > 0.
\end{align}
If this dual problem admits a unique solution $Y^\ast_T \in \cY$, then the primal problem \eqref{eq:opti} with $H=0$ also yields a unique solution
\begin{align*}
X^{\pi^\ast}_T &= x + \int_0^T X^{\pi^\ast}_s \pi^\ast_s  \frac{d\tilde{S}_s}{\tilde{S}_s}\\
&= x + \int_0^T \alpha^\ast_s dS_s \\
&= I(yY^\ast_T), 
\end{align*}
with the corresponding optimal control $\pi^\ast = \frac{\alpha^\ast \tilde{S}}{X^{\pi^\ast}}$. Here we have $I = (U')^{-1}$ and $x = - v'(y)$\footnote{This is equivalent to $u'(x) = y$ where $u(x) = \sup_\pi \E\big[ U(X^\pi_T + H) \big]$. The differentiability of both $v(y)$ and $u(x)$ are shown in \cite{CvitanicSchachermayerWang}.}. The case of bounded terminal endowment $H$ is dealt with in \cite{CvitanicSchachermayerWang}, where instead of \eqref{eq:dual_problem1} the following dual problem is considered
\begin{align*}
v(y) = \inf_{Y_T \in \cY} \E \big[ V(yY_T) + y Y_T H \big] , \quad y > 0.\end{align*}
The case of general integrable $H$ has been studied in \cite{HugonnierKramkov}, using the original dual problem \eqref{eq:dual_problem1} but a slight different choice of the domain $\cY$. A ubiquitous property of the convex duality method is that once the primal and the dual optimizers are obtained, their product $X^{\pi^\ast} Y^\ast$ is a nonnegative true martingale (hence uniformly integrable), see \cite{KramkovSchachermayer} for a economic interpretation.
In the context of utility maximization with bounded random endowments, this martingale property of $X^{\pi^\ast} Y^\ast$ is pointed out in \cite[Remark 4.6]{CvitanicSchachermayerWang}.
This martingale property of $X^{\pi^\ast} Y^\ast$ constitutes the first main ingredient for deriving a solution for the forward-backward equation \eqref{eq:fbsde3bis}. A second main ingredient is constituted by the characterization of the dual domain $\cY$. Note in the continuous process setting, $\cY$ is the family of all non-negative supermartingales (see e.g. \cite{KramkovSchachermayer,HugonnierKramkov}). According to a well known result, every nonnegative c\`adl\`ag supermartingale $Y \in \cY$ admits a unique multiplicative decomposition $$Y = A M$$ where $A$ is a predictable, non-increasing process such that $A_0 = 1$ and $M$ is c\`adl\`ag local martingale. However, \cite{LarsenZitkovic} characterize the elements of $Y \in \cY$ by the multiplicative decomposition
\begin{align}
\label{eq:LarsenZitkovic}
Y = A \cE(-\theta^\cH \cdot W^\cH + K \cdot W^\cO),
\end{align}
where $A$ is a predictable non-increasing process such that $A_0 = 1$ and $K \in \mathbb{H}^2_{loc}(\real^{d_2})$ (see \cite[Proposition 3.2]{LarsenZitkovic}). Using that the Fenchel-Legendre transform $V$ is strictly decreasing, \cite[Corollary 3.3]{LarsenZitkovic} shows that the dual optimizer is a (continuous) local martingale and admits the representation
\begin{align}
\label{eq:LarsenZitkovic2}
Y^\ast = \cE \big( - \theta^\cH \cdot W^\cH + K^\ast \cdot W^\cO \big)
\end{align}
for a uniquely determined $K^\ast \in \mathbb{H}^2_{loc}(\real^{d_2})$. If $v(y) = \E \Big[ V (yY^\ast_T) \Big] < \infty$, then we can check that the optimal $K^\ast$ actually belongs to $\mathbb{H}^2(\real^{d_2})$. This is done in the following lemma whose proof is in the same spirit as in \cite[Lemma 3.2]{Larsen}
\begin{lemma}
\label{lemma:Z2_H2}
If for some $y>0$, it holds that
\begin{align*}
{v}(y) &= \inf_{\nu\in \mathbb{H}_{loc}^2(\real^{d_2})}\E \Big[ {V} \Big( y \cE\big( -\theta^\cH \cdot W^\cH + \nu \cdot W^\cO \big) \Big)  \Big] < \infty,
\end{align*}
we have
\begin{align*}
{v}(y) &= \inf_{\nu\in\mathbb{H}^2(\real^{d_2})}\E \Big[ {V} \Big( y \cE\big( -\theta^\cH \cdot W^\cH + \nu \cdot W^\cO \big) \Big)  \Big],
\end{align*}
i.e. the optimal $K^\ast$ minimizing $v(y)$ can be assumed to belong to $\mathbb{H}^2(\real^{d_2})$.
\end{lemma}
\begin{proof}
We introduce the family of stopping times
\begin{align*}
\tau^n := \inf \big\{ t>0: \int_0^t \big(|\theta^\cH_s|^2 + |K^\ast_s|^2 \big) ds \geq n \big\}, ~ n \in \N.
\end{align*}
Let $y>0$, then we have
\begin{align*}
{v}(y) &= \E \Big[ {V} \Big( y \cE_T\big( -\theta^\cH \cdot W^\cH + K^\ast \cdot W^\cO \big) \Big) \Big]\\
&= \E \left[ \E \Big[ {V} \Big( y\cE_T\big( -\theta^\cH \cdot W^\cH + K^\ast \cdot W^\cO \big) \Big) \big| \cF_{\tau^n}\Big] \right]\\
&\geq \E \Big[ {V} \Big( y \cE_{\tau^n}\big( -\theta^\cH \cdot W^\cH + K^\ast \cdot W^\cO \big) \Big) \Big],
\end{align*}
where the last line follows by Jensen's inequality. Continuing the last line and recalling that ${V}(y)$ is a strictly convex function, we have
\begin{align*}
{v}(y) &\geq \E \Big[ {V} \Big( y \exp\big( \int_0^{\tau^n} \big( -\theta^\cH_s dW^\cH_s + K^\ast_s dW^\cO_s \big) \big) \exp\big( -\frac12 \int_0^{\tau^n} \big( |\theta^\cH_s|^2 + |K^\ast_s|^2 \big) ds \big)  \Big)\Big]\\
&\geq {V} \left( \E \Big[ y \exp\big( \int_0^{\tau^n} \big( -\theta^\cH_s dW^\cH_s + K^\ast_s dW^\cO_s \big) \big) \exp\big( -\frac12 \int_0^{\tau^n} \big( |\theta^\cH_s|^2 + |K^\ast_s|^2 \big) ds \big) \Big] \right)\\
&\geq {V} \left( y\exp\Big( \E \Big[ -\frac12 \int_0^{\tau^n} \big( |\theta^\cH_s|^2 + |K^\ast_s|^2 \big) ds \Big] \Big) \right),
\end{align*}
where Jensen's inequality has been used twice. By continuity of ${V}$ and of the exponential function, it follows from the monotone convergence theorem that
\begin{align*}
{v}(y) &\geq \lim_{n\to\infty} {V} \Big( \exp\Big(  -\frac12 \E\Big[ \int_0^{\tau^n} \big( |\theta^\cH_s|^2 + |K^\ast_s|^2 \big) ds \Big] \Big) \Big)\\
&= {V} \Big( \exp\Big(  -\frac12 \E\Big[ \int_0^T \big( |\theta^\cH_s|^2 + |K^\ast_s|^2 \big) ds \Big] \Big) \Big).
\end{align*}
Since ${v}(y) < \infty$ and ${V}\big(\exp(-\infty) \big) = {V}(0) = U(\infty) = \infty$, it follows that
\begin{align*}
\E \Big[ \int_0^T \big( |\theta^\cH_s|^2 + |K^\ast_s|^2\big) ds \Big] < \infty.
\end{align*}
We deduce that $K^\ast \in \mathbb{H}^2(\real^{d_2})$.
\end{proof}

Now using that $X^{\pi^\ast} Y^\ast$ is a true martingale and that the dual optimizer $Y^\ast$ is a local martingale satisfying \eqref{eq:LarsenZitkovic2}, we get the following result.
\begin{theorem}
\label{thm:H2}
Let $H$ be a non-negative bounded random endowment and assume that the coefficient of relative risk aversion $- \frac{xU^{''}(x)}{U^{'}(x)}$ satisfies
\begin{align}
\label{eq:rel_risk_aversion}
\limsup_{x\to\infty} \left( - \frac{xU^{''}(x)}{U^{'}(x)} \right) < \infty.
\end{align}
Then there exists $x_0 >0$ such that for all $x>x_0$ the coupled FBSDE \eqref{eq:fbsde3} has a solution $(X,Y,Z)$ such that $X_0 =x$. In addition, $X$ is the optimal wealth of the problem \eqref{eq:opti} and the dual optimizer $Y^\ast$ associated with it is given by $Y^\ast=U'(X) \exp(Y)$ (so that $y Y_T^\ast=U'(X_T+H)$).
\end{theorem}
\begin{proof}
The existence of $x_0> 0$ such that for every $x > x_0$ the quantity
\begin{align*}
u(x) = \sup_{\pi \in \Pi^x} \E \big[ U\big( X^\pi_T + H \big) \big] = \E \big[ U\big( X^{\pi^\ast}_T + H \big) \big]
\end{align*}
is finite has been shown \cite{CvitanicSchachermayerWang}. We set $X^\ast := X^{\pi^\ast}$. Also recall that we have $y = u'(x)>0$ for $x>x_0$ and that we have $$ \E \Big[ y X^\ast_T Y^\ast_T \Big] = xy. $$ 
Moreover, $y Y_T^\ast=U'(X_T^\ast+H)$. We define the true martingale $\alpha:=y X^\ast Y^\ast$.
We set $Y:=\log(\alpha) - \log(X^\ast) - \log(U'(X^\ast))$. We have that
\begin{align*}
Y_t&=\log\left(\frac{\alpha_t}{X_t^\ast U'(X_t^\ast)}\right)\\
&= \log\left(\frac{y Y_t^\ast}{U'(X_t^\ast)}\right)\\
&= \log(y) + \log(Y_t^\ast)-\log(U'(X_t^\ast).
\end{align*}
Recall that by definition of $X^\ast$ and $Y^\ast$ we have that
$$ dY_t^\ast= Y_t^\ast \left( - \theta_t^\H dW_t^\H + K_t^\ast dW_t^\mathcal{O} \right) $$
and
$$ d X_t^\ast = X_t^\ast \left( \pi_t^\ast dW_t^\H + \pi_t^\ast \theta_t^\H dt \right).$$
Hence
\begin{align*}
dY_t &= -\theta_t^\H dW_t^\H + K_t^\ast dW_t^\mathcal{O} -\frac12 (|\theta_t^\H|^2+|K_t^\ast|^2) dt\\
&\quad - \frac{U''(X_t^\ast)}{U'(X_t^\ast)} (\pi_t^\ast X_t^\ast dW_t^\H + \pi_t^\ast X_t^\ast \theta_t^\H dt)\\
&\quad - \frac12 \frac{U^{(3)}(X_t^\ast)U'(X_t^\ast)-(U''(X_t^\ast))^2}{(U'(X_t^\ast))^2} |\pi_t^\ast X_t^\ast |^2 dt.  
\end{align*}
We define:
$$ Z_t^\H := -\theta_t^\H - \frac{U''(X_t^\ast)}{U'(X_t^\ast)} \pi_t^\ast X_t^\ast,$$
so that $\pi_t^\ast X_t^\ast = -(Z_t^\H+\theta_t^\H) \frac{U'(X_t^\ast)}{U''(X_t^\ast)}$, and
$$ Z_t^\mathcal{O}:= K_t^\ast.$$
Then
\begin{align*}
dY_t &= Z_t^\H dW_t^\H + Z_t^\mathcal{O} dW_t^\mathcal{O} -\frac12 (|\theta_t^\H|^2+|K_t^\ast|^2) dt\\
&\quad + \left[ \theta_t^\H (Z_t^\H+\theta_t^\H) -\frac12 \frac{U^{(3)}(X_t^\ast)U'(X_t^\ast)-(U''(X_t^\ast))^2}{(U'(X_t^\ast))^2} \frac{|Z_t^\H+\theta_t^\H|^2 |U'(X_t^\ast)|^2}{|U''(X_t^\ast)|^2} \right] dt\\
&= Z_t^\H dW_t^\H + Z_t^\mathcal{O} dW_t^\mathcal{O} + \left[ |Z_t^\H+\theta_t^\H|^2 \left(1-\frac12 \frac{U^{(3)}(X_t^\ast)U'(X_t^\ast)}{|U''(X_t^\ast)|^2}\right)-\frac12 |Z_t^\H|^2 \right] dt.
\end{align*}
Finally note that by construction $Y_T=\log\left(\frac{U'(X_T^\ast+H)}{U'(X_T\ast)}\right)$.
Hence, $(X,Y,Z) = (X^\ast,Y,Z)$ is a solution to \eqref{eq:fbsde3} and $$ y Y^\ast = U'(X) \exp(Y). $$
\end{proof}

Let us recall that the \textit{absolute risk aversion} of $U(x)$ is defined as $ARA(x) := - \frac{U^{''}(x)}{U^{'}(x)}$ and the \textit{risk tolerance} as $\frac{1}{ARA(x)}$. We say that $U(x)$ has \textit{hyperbolic absolute risk aversion (HARA)} if and only if its risk tolerance $\frac{1}{ARA(x)}$ is linear in $x$. More precisely, it can be shown that a utility function $U(x)$ is \textit{HARA} if and only if $$U(x)= \frac{1-\gamma}{\gamma}\Big( \frac{a x}{1-\gamma} + b \Big)^\gamma, ~~ \frac{a x}{1-\gamma} + b >0,$$ for given real numbers $\gamma, a, b \in \IR$.
\begin{corollary}
Assume that $U(x)$ is \textit{HARA}. Then there exists a constant $\kappa \in \IR$ such that the backward equation from \eqref{eq:fbsde3} can be written as
\begin{align}
\label{eq:temp002}
Y_t &= \log\Big( \frac{U^{'}(X^\ast_T + H)}{U^{'}(X^\ast_T)} \Big) - \int_t^T Z_s dW_s - \int_t^T \Big( -\frac12 |Z_s|^2 + \kappa |Z^\cH_s + \theta^\cH_s|^2 \Big) ds\\
&= \log\Big( \frac{U'(X^\ast_T + H)}{U^{'}(X^\ast_T)} \Big) - \int_t^T Z_s dW_s - \int_t^T g(s,Z_s) ds.\nonumber
\end{align}
\end{corollary}
\begin{proof}
Notice that for the risk tolerance $$f(x):= \frac{1}{ARA(x)} = -\frac{U^{'}(x)}{U^{''}(x)}$$ it holds that $$ f^{'}(x) = -1 + \frac{U^{'}(x) U^{(3)}(x)}{| U^{''}(x) |^2}.$$ Since $U(x)$ being \textit{HARA} implies that $f$ is linear in $x$, it follows that there exist constants $c,d \in \IR$ such that $f^{'}(x) = c x + d$. Hence the BSDE from \eqref{eq:fbsde3} can also be written as
\begin{align*}
Y_t &= \log\Big( \frac{U^{'}(X^\ast_T + H)}{U^{'}(X^\ast_T)} \Big) - \int_t^T Z_s dW_s - \int_t^T \Big( -\frac12 |Z_s|^2 + \big( \frac12 -\frac12 f^{'}(X^\ast_s) \big)|Z^\cH_s + \theta^\cH_s|^2 \Big) ds\\
&= \log\Big( \frac{U^{'}(X^\ast_T + H)}{U^{'}(X^\ast_T)} \Big) - \int_t^T Z_s dW_s - \int_t^T \Big( -\frac12 |Z_s|^2 + \kappa |Z^\cH_s + \theta^\cH_s|^2 \Big) ds,
\end{align*}
for $\kappa = \frac12 - \frac12 c$.
\end{proof}
Obviously the driver of the BSDE \eqref{eq:temp002}, $g(s,z)$, satisfies the quadratic growth condition
\begin{align*}
|g(s,z)| \leq \alpha + \frac\gamma2 |z|^2
\end{align*}
for suitably chosen real numbers $\alpha, \gamma >0$. In this setting  \cite[Theorem 2]{BriandHu1} yields the following result.
\begin{corollary}\label{corollary:BriandHu}
If $\xi = \log\Big( \frac{U^{'}(X^\ast_T + H)}{U^{'}(X^\ast_T)} \Big)$ satisfies $\E \Big[ e^{\gamma |\xi|} \Big] < \infty$, then the BSDE \eqref{eq:temp002} admits a solution $(Y,Z)$ such that $Y$ is continuous and $Z\in \mathbb{H}^2_{loc}(\real^d)$.
\end{corollary}

\subsection{The power case with general endowment}
\label{section:power}

We finally deal with an open question in mathematical Finance namely the case of power utility with general endowment. We know from duality theory that an optimal solution exists but we would like to prove that the strategy is smooth (i.e. square integrable) and to characterize it in terms of the solution to an equation (for instance a FBSDE). We will use definitions and notations of Section \ref{section:halfline}. Let $U(x):=\frac{x^\gamma}{\gamma}$ with $\gamma$ a fixed parameter in $(0,1)$. Let $H$ be a positive bounded $\mathcal{F}_T$-measurable random variable where we recall that $(\mathcal{F}_t)_{t\in [0,T]}$ is the filtration generated by $W=(W^\cH,W^\cO)$. We recall that we denote by $\Pi^x$ the set of admissible strategies with initial capital $x$ which is now defined by
\begin{equation}
\label{eq:admissiblebis}
\Pi^x:=\L\{ \pi:\Omega \times [0,T] \to \real^{d_1}, \; \pi \mbox{ is predictable}, \;  \E\L[\int_0^T |\pi_s|^2 ds\R]<\infty \R\}
\end{equation}
where $\pi^i, i=1,\ldots,d_1$ denotes the proportion of wealth invested in the stock. The associated wealth process is given by
$$ X_t^\pi:=x+\int_0^t \pi_s X_s^{\pi} dS_s^{\mathcal{H}}, \quad t\in [0,T].$$
Again, we extend $\pi$ to $\real^d$ via $\tilde{\pi}:=(\pi^1,\ldots,\pi^{d_1},0,\ldots,0)$ and make the convention that we write $\pi$ instead of $\tilde{\pi}$.
Thus, we have
$$ X_t^\pi=x \mathcal{E}\left(\int_0^\cdot \pi_r dS_r^{\mathcal{H}} \right)_t, \quad t \in [0,T].$$
Note that this setting covers the case of a purely orthogonal endowment of the form $H:=\phi(S_T^\cO)$ where $\phi$ is positive. Now we can go in the analysis of the problem:
\begin{equation}
\label{eq:optipowerincomstrat}
\sup_{\pi \in \Pi^x} \E\left[\frac{(X_T^\pi+H)^\gamma}{\gamma}\right].
\end{equation}
Indeed, what is only known in that case is that an optimal strategy exists (\cite{HugonnierKramkov}) but in a much larger space that $\Pi^x$, in particular it is not proved that the optimal strategy is square integrable. About the characterization of this optimal strategy one can write the Hamilton-Jacobi-Bellman PDE in the Markovian case but no results allow us to solve it. We believe that combining the duality theory, BSDEs techniques and our approach we could show first that the optimal strategy belongs to the space $\Pi^x$ and that we will give a characterization of it in terms of a FBSDE. Let us be more precise.
\begin{theorem}
\label{th:powergeneral}
There exists $x_0>0$ such that for every $x>x_0$, the system
\begin{equation}
\label{eq:powerincomp}
\left\lbrace
\begin{array}{l}
X_t=x+\int_0^t  \frac{X_s (Z_s^\cH+\theta_s^\cH)}{1-\gamma} dW_s^\cH+\int_0^t \theta_s^\cH \frac{X_s (Z_s^\cH+\theta_s^\cH)}{1-\gamma} ds \\
\\
Y_t=(\gamma-1) \log\left(1+\frac{H}{X_T}\right)-\int_t^T Z_s dW_s -\int_t^T \left(\frac{\gamma}{2 (\gamma-1)} |Z_s^\cH+\theta_s^\cH|^2 - \frac12 |Z_s|^2 \right) ds
\end{array}
\right.
\end{equation}
admits an adapted solution $(X,Y,Z)$. If in addition $Z^\cH=(Z^1,\ldots,Z^{d_1})$ is in $\mathbb{H}^2(\real^{d_1})$, then
\begin{equation}
\label{eq:optipower}
{\pi^\ast}^i:=\frac{1}{1-\gamma} (Z^i+\theta^i), \; i=1,\ldots,d_1
\end{equation}
is the optimal solution to the maximization problem \eqref{eq:optipowerincomstrat}.
\end{theorem}

\begin{proof}
First note that the system \eqref{eq:powerincomp} is exactly the system \eqref{eq:fbsde3} with $U(x)=\frac{x^\gamma}{\gamma}$. Hence from Theorem \ref{thm:H2} there exists $x_0>0$ such that the system \eqref{eq:powerincomp} admits a solution $(X,Y,Z)$ when $x>x_0$. We fix, $x>x_0$ and consider the associated solution $(X,Y,Z)$ (that is $X_0=x$). In addition, we know from Theorem \ref{thm:H2} that $X=X^\ast$. Hence $\pi^\ast$ is given by \eqref{eq:optipower}. It just remains to prove that $\pi^\ast$ is in $\Pi^x$, which is a direct consequence of the fact that $Z$ is in $\mathbb{H}^2(\real^d)$.
\end{proof}

\begin{remark}
Note that since we know that the dual optimizer $Y^\ast$ is given by $Y^\ast=U'(X) \exp(Y)$ it is clear that $X U'(X) \exp(Y)$ is a true martingale. Hence the square integrability of $Z$ implies the condition of Theorem \ref{th:conversehalfline}: $\E[(X_T+H)^\gamma]<\infty$. Finally notice that $Z^\cO$ is in $\mathbb{H}^2(\real^{d_2})$ by Lemma \ref{lemma:Z2_H2}.
\end{remark}

So the only element missing in the proof is indeed to show that $Z^\cH$ is in $\mathbb{H}^2(\real^{d_1})$ (naturally, since the process $\pi^\ast$ is integrable with respect to $S^{\cH}$ and so it is in $\mathbb{H}^2(\real^{d_1})$). This question requires a deeper analysis of the system and is currently investigated by the authors.

\section*{Acknowledgments}
Horst acknowledges financial support through the SFB 649 \emph{Economic Risk}. Imkeller and R\'eveillac are grateful to the \emph{DFG Research Center MATHEON, Project E2}.
Hu is partially supported by the Marie Curie ITN Project \emph{Deterministic and Stochastic Controlled Systems and Applications}, call: F97-PEOPLE-2007-1-1-ITN, n. 213841-2. Zhang acknowledges support by \emph{DFG IRTG 1339 SMCP}.



\begin{thebibliography}{10}

\bibitem{Biagini}
S.~Biagini.
\newblock Expected utility maximization: the dual approach.
\newblock {\em Encyclopedia of Quantitative Finance}, 2009.

\bibitem{Bismut}
J.-M. Bismut.
\newblock Th\'eorie probabiliste du contr\^ole des diffusions.
\newblock {\em Mem. Am. Math. Soc.}, 167, 1976.

\bibitem{Bismut78}
J.-M. Bismut.
\newblock An introductory approach to duality in optimal stochastic control.
\newblock {\em SIAM Rev.}, 20(1):62--78, 1978.

\bibitem{BriandHu1}
P.~Briand and Y.~Hu.
\newblock {BSDE with quadratic growth and unbounded terminal value.}
\newblock {\em Probab. Theory Relat. Fields}, 136(4):604--618, 2006.

\bibitem{CvitanicSchachermayerWang}
J.~Cvitani\'c, W.~Schachermayer, and H.~Wang.
\newblock Utility maximization in incomplete markets with random endowment.
\newblock {\em Finance Stoch.}, 5(2):259--272, 2001.

\bibitem{DelbaenPengRosazzaGianin}
F.~Delbaen, S.~Peng, and E.~Rosazza~Gianin.
\newblock Representation of the penalty term of dynamic concave utilities.
\newblock {\em Finance. Stoch.}, 14(3):449--472, 2010.

\bibitem{DelbaenSchachermayer}
F.~Delbaen and W.~Schachermayer.
\newblock {\em The mathematics of arbitrage}.
\newblock Springer-Finance. Berlin: Springer., 2006.

\bibitem{HorstPirvuReis}
U.~Horst, T.~Pirvu, and G.~Dos~Reis.
\newblock On securitization, market completion and equilibrium risk transfer.
\newblock {\em Math. Finan. Econ.}, 2(4):211--252, 2010.

\bibitem{HuImkellerMueller}
Y.~Hu, P.~Imkeller, and M.~M{\"u}ller.
\newblock Utility maximization in incomplete markets.
\newblock {\em Ann. Appl. Probab.}, 15(3):1691--1712, 2005.

\bibitem{HugonnierKramkov}
J.~Hugonnier and D.~Kramkov.
\newblock {Optimal investment with random endowments in incomplete markets.}
\newblock {\em Ann. Appl. Probab.}, 14(2):845--864, 2004.

\bibitem{ImkellerReveillacZhang}
P.~Imkeller, A.~R{\'e}veillac, and J.~Zhang.
\newblock Solvability and numerical simulation of bsdes related to bspdes with
  applications to utility maximization.
\newblock {\em To appear in Int. J. Theor. Appl. Finance}, 2011.

\bibitem{Kobylanski}
M.~Kobylanski.
\newblock Backward stochastic differential equations and partial differential
  equations with quadratic growth.
\newblock {\em Ann. Probab.}, 28(2):558--602, 2000.

\bibitem{KramkovSchachermayer}
D.~Kramkov and W.~Schachermayer.
\newblock The asymptotic elasticity of utility functions and optimal investment
  in incomplete markets.
\newblock {\em Ann. Appl. Probab.}, 9(3):904--950, 1999.

\bibitem{Larsen}
K.~Larsen.
\newblock {A note on the existence of the power investor's optimizer.}
\newblock To appear in Finance Stoch., 2009.

\bibitem{LarsenZitkovic}
K.~Larsen and G.~\v{Z}itkovi\'c.
\newblock {Stability of utility-maximization in incomplete markets.}
\newblock {\em Stochastic Processes Appl.}, 117(11):1642--1662, 2007.

\bibitem{MaZhang}
J.~Ma and J.~Zhang.
\newblock Representation theorems for backward stochastic differential
  equations.
\newblock {\em Ann. Appl. Probab.}, 12(4):1390--1418, 2002.

\bibitem{ManiaSantacroce}
M.~Mania and M.~Santacroce.
\newblock Exponential utility optimization under partial information.
\newblock {\em Finance Stoch.}, 14(3):419--448, 2010.

\bibitem{ManiaTevzadze}
M.~Mania and R.~Tevzadze.
\newblock Backward stochastic {PDE}s related to the utility maximization
  problem.
\newblock {\em Georgian Math. J.}, 17(4):705--740, 2010.

\bibitem{MochaWestray}
M.~Mocha and N.~Westray.
\newblock Sensitivity analysis for the cone constrained utility maximization
  problem.
\newblock {\em Preprint}, 2010.

\bibitem{Morlais}
M.-A. Morlais.
\newblock Quadratic {BSDE}s driven by a continuous martingale and applications
  to the utility maximization problem.
\newblock {\em Finance Stoch.}, 13(1):121--150, 2009.

\bibitem{Nutz}
M.~Nutz.
\newblock The opportunity process for optimal consumption and investment with
  power utility.
\newblock {\em Math. Finan. Econ.}, 3(3):139--159, 2010.

\bibitem{Peng}
S.~Peng.
\newblock Backward stochastic differential equations and applications to
  optimal control.
\newblock {\em Appl. Math. Optimization}, 27(2):125--144, 1993.

\bibitem{Protter}
P.~Protter.
\newblock {\em Stochastic integration and differential equations. 2nd ed.}
\newblock Number~21 in Applications of Mathematics. Springer Berlin, 2004.

\end{thebibliography}
\end{document}